\documentclass[sn-mathphys-ay]{sn-jnl}

\usepackage{graphicx}%
\usepackage{multirow}%
\usepackage{amsmath,amssymb,amsfonts}%
\usepackage{amsthm}%
\usepackage{mathrsfs}%
\usepackage[title]{appendix}%
\usepackage{xcolor}%
\usepackage{textcomp}%
\usepackage{manyfoot}%
\usepackage{booktabs}%
\usepackage{algorithm}%
\usepackage{algorithmicx}%
\usepackage{algpseudocode}%
\usepackage{listings}%
\usepackage{subcaption}
\usepackage{color}
\usepackage{microtype}
\usepackage{mathtools}
\usepackage{lmodern}
\usepackage{fix-cm}
\usepackage{scalefnt}

\newtheorem{theorem}{Theorem}[section]
\newtheorem{proposition}{Proposition}[section]
\newtheorem{corollary}{Corollary}[section]
\newtheorem{lemma}{Lemma}[section]

\newtheorem{example}{Example}%
\newtheorem{remark}{Remark}%

\newtheorem{definition}{Definition}%

\newcommand{\N}{\mathbb{N}}

\newcommand{\I}{\mathbb{I}}
\renewcommand{\P}{\mathbb{P}}
\newcommand{\E}{\mathbb{E}}
\newcommand{\R}{\mathbb{R}}

\newcommand{\AssumpAOneText}{(A1)}
\newcommand{\AssumpAOne}{\hyperref[assumpt:AssumpAOne]{\AssumpAOneText}}
\newcommand{\AssumpATwoText}{(A2)}
\newcommand{\AssumpATwo}{\hyperref[assumpt:AssumpATwo]{\AssumpATwoText}}
\newcommand{\AssumpAThreeText}{(A3)}
\newcommand{\AssumpAThree}{\hyperref[assumpt:AssumpAThree]{\AssumpAThreeText}}
\newcommand{\AssumpVelocityText}{(A4)}
\newcommand{\AssumpVelocity}{\hyperref[assumpt:AssumpVelocity]{\AssumpVelocityText}}
\newcommand{\AssumpVelocityPrimeText}{(A4')}

\newcommand{\smallSet}{K}

\begin{document}

\title[Tail Index Estimation for Discrete Heavy-Tailed Distributions]{Tail Index Estimation for Discrete Heavy-Tailed Distributions with Application to Statistical Inference for Regular Markov Chains}


\author[1]{\fnm{Patrice} \sur{Bertail}}
\email{patrice.bertail@parisnanterre.fr}
\equalcont{These authors contributed equally to this work.}
\author[2]{\fnm{Stephan} \sur{Cl\'emen\c{c}on}}
\email{stephan.clemencon@telecom-paris.fr}
\equalcont{These authors contributed equally to this work.}
\author*[1,2]{\fnm{Carlos} \sur{Fern\'andez}}
\email{fernandez@telecom-paris.fr}
\equalcont{These authors contributed equally to this work.}
\affil[1]{\orgdiv{MODAL'X, UMR CNRS 9023}, \orgname{Universit\'e Paris Nanterre},
    \orgaddress{\street{200 Avenue de la République}, \city{Nanterre}, \postcode{92000}, \state{Ile-de-France},
        \country{France}}
}

\affil*[2]{\orgdiv{LTCI, Telecom Paris}, \orgname{Institut Polytechnique de Paris},
    \orgaddress{\street{19 place Marguerite Perey}, \city{Palaiseau}, \postcode{91123}, \state{Ile-de-France},
        \country{France}}
}


\abstract{
    It is the purpose of this paper to investigate the issue
    of	estimating the regularity index $\beta>0$ of a discrete
    heavy-tailed r.v. $S$, \textit{i.e.} a r.v. $S$ valued
    in $\mathbb{N}^*$ such that $\mathbb{P}(S>n)=L(n)\cdot n^{-\beta}$
    for all $n\geq 1$, where $L:\mathbb{R}^*_+\to \mathbb{R}_+$
    is a slowly varying function. Such discrete probability laws, referred to as
    generalized Zipf's laws sometimes, are commonly used to
    model rank-size distributions after a preliminary range
    segmentation in a wide variety of areas such as
    \textit{e.g.} quantitative linguistics, social sciences
    or information theory. As a first go, we consider the
    situation where inference is based on independent copies
    $S_1,\; \ldots,\; S_n$ of the generic variable $S$. Just
    like the popular Hill estimator in the continuous
    heavy-tail situation, the estimator $\widehat{\beta}$ we
    propose can be derived by means of a suitable
    reformulation of the regularly varying condition,
    replacing $S$'s survivor function by its empirical
    counterpart. Under mild assumptions, a non-asymptotic
    bound for the deviation between $\widehat{\beta}$ and
    $\beta$ is established, as well as limit results
    (consistency and asymptotic normality). Beyond the
    i.i.d. case, the inference method proposed is extended
    to the estimation of the regularity index of a
    regenerative $\beta$-null recurrent Markov chain. Since
    the parameter $\beta$ can be then viewed as the tail
    index of the (regularly varying) distribution of the
    return time of the chain $X$ to any (pseudo-)
    regenerative set, in this case, the estimator is
    constructed from the successive regeneration times.
    Because the durations between consecutive regeneration
    times are asymptotically independent, we can prove that
    the consistency of the estimator promoted is preserved.
    In addition to the theoretical analysis carried out,
    simulation results provide empirical evidence of the
    relevance of the inference technique proposed.}

\maketitle

\section{Introduction}\label{sec1}

This article is devoted to the study of the problem of estimating the regularity index
$\beta>0$ of a generalized discrete Pareto distribution, namely the probability
distribution of a random variable $S$ defined on a probability space $(\Omega,\;
    \mathcal{F},\; \mathbb{P})$, taking its values in $\mathbb{N}^*$ and such that:
\begin{equation}\label{eq:Zipf}
    \P\left( {S > n} \right) = {n^{ - \beta }}L\left( n \right) \text{ for all } n\geq 1,
\end{equation}
where $L:\mathbb{R}_+\to \mathbb{R}$ is a slowly varying function, \textit{i.e.} such
that $L(\lambda z)/L(z)\rightarrow +1$ as $z\rightarrow +\infty$ for any $\lambda>0$, see
\cite{Bingham1987}.
Such discrete power law probability distributions, also referred to
as generalized Zipf's laws sometimes, are often used to model the distribution of
discrete data exhibiting a specific rank-frequency relationship, namely when the
logarithm of the frequency and that of the rank order are nearly proportional. Such a
phenomenon has been empirically observed in many ranking systems: in quantitative
linguistics (\textit{i.e.} when analysing word frequency law in natural language, see
\textit{e.g.} \cite{Manning}) in the first place, as well as in a very wide variety of
situations, too numerous to be exhaustively listed here. One may refer to \cite{SSBN18},
\cite{Lazzardi2021.06.16.448706} or \cite{Zanette} among many others. In this paper, we
first consider the issue of estimating the parameter $\beta$ involved in \eqref{eq:Zipf}
(supposedly unknown, like the function $L$) in the classic (asymptotic) i.i.d.
statistical setting, \textit{i.e.} based on an increasing number $n\geq 1$ of independent
copies $S_1,\; \ldots,\; S_n$ of the generic r.v. $S$. Statistical inference for discrete
heavy-tailed distributions has not received much attention in the literature and still poses methodological problems that this article seeks to resolve. Most of the
very few dedicated methods documented either deal with very specific cases as in
\textit{e.g.} \cite{goldstein2004}, \cite{Matsui2013} and \cite{clauset2009power} or else
consist in empirically applying techniques originally designed for continuous heavy-tailed
distributions to discrete data after first adding independent uniform noise \citep{Voitalov2019,kim2020consistency, kim2023asymptotic}.

The vast majority of the regular variation index estimators proposed in the literature,
Hill or Pickand estimators in particular \citep{Hill1975,Pickands1975}, are based on
order statistics, which causes obvious difficulties in the discrete case because of the
possible occurrence of many ties. Indeed, in that case, many spacings are equal to $0$
which causes these estimators to behave erroneously. This is stressed in \cite{Matsui2013}
(see their figures 1 and 2). More importantly, \cite{Matsui2013} constructs an example of
a family of discrete value heavy-tailed distributions that do not satisfy any
second-order condition so that the asymptotic normality of the Hill estimator cannot be
proved directly. In contrast, the estimator under study here is based on the analysis of
the probability of exponentially separated tail events. It simply rests on the fact that,
as can be immediately deduced from \eqref{eq:Zipf}, we have
$\ln(p_k)-\ln(p_{k+1})=\beta+\ln(L(e^{k})/L(e^{k+1}))$, where $\ln(x)$ denotes the
natural logarithm of any real number $x>0$ and $p_l=\mathbb{P}(S>e^{l})$ for all $l\in
    \mathbb{N}$, and that $L(e^{k + 1})/L(e^k)$ is expected to be very close to $1$ for $k\in
    \mathbb{N}$ chosen sufficiently large. A natural (plug-in) inference technique can be
then devised by replacing the tail probabilities $p_l$ with their empirical versions
$\widehat{p}^{(n)}_l=(1/n)\sum_{i=1}^n \mathbb{I}\{S_i>e^l\}$ for $l\in \mathbb{N}$,
where $\mathbb{I}\{\mathcal{A}\}$ means the indicator function of any event
$\mathcal{A}$. This yields the estimator
\begin{equation}\label{eq:est}
    \widehat{{\beta }}_n\left( k \right) = \ln \left( {\widehat p_k^{(n)}} \right) - \ln \left( \widehat p_{k + 1}^{(n)} \right),
\end{equation}
provided that $\widehat{p}_{k + 1}^{(n)}>0$ (as shall be seen, this occurs with large
probability if $n$ is sufficiently large). By convention, we set $\widehat{{\beta
            }}_n( k )=0$ when $\widehat{p}_{k + 1}^{(n)}=0$. We point out that it has
exactly the same form as that proposed and analysed in \cite{Carpentier2015} in a
different context, that of (continuous) \textit{approximately Pareto
    distributions}\footnote{The distribution of a real-valued r.v. $X$ is said to be
    \textit{approximately Pareto} with tail index $\beta>0$ if its survivor function is of
    the form: $\forall x>0$, $\mathbb{P}(X>x)=L(x)x^{-\beta}$, where $L$ is asymptotically
    constant at infinity, \textit{i.e.} there exists $C\in (0,\; \infty)$ s.t. $L(x)\to C$ as
    $x\rightarrow +\infty$.} namely. In the discrete generalized Pareto framework, we prove
that for an appropriate choice of the hyper-parameter \(k=k_n\) (typically chosen of order
\(\ln(n)\)), the estimator \eqref{eq:est} is strongly consistent and asymptotically
normal as $n\rightarrow +\infty$. We also show that  non-asymptotic upper confidence bounds for the absolute
deviations between $\widehat{{\beta }}_n\left( k \right)$ and $\beta$ are also
established here.

Although estimation of the parameter $\beta$ in \eqref{eq:Zipf} in the discrete i.i.d.
setting is an important issue in itself, the present paper also finds its motivation in
the problem of recovering statistically the regularity index of a
regenerative \textit{regular} $\beta$-null-recurrent Markov chain $X=(X_n)_{n\in
            \mathbb{N}}$, based on the observation of a finite sample path $X_1,\; \ldots,\; X_n$
with $n\geq 1$. As explained in \cite{Chen1999, Chen2000}, for regular
Markov chains, the regularity index \(\beta\) controls the (sublinear) rate at
which the number of visits to any given Harris set increases with observation time $n$,
no matter the initial distribution. These Markov chains are strongly
non-stationary and are not mixing in any sense. Typical examples of these chains
are random walks and their variations, for
instance Bessel random walks (see \cite{Alexander2011}). Such
processes appear also naturally in many applications: for population dynamics or
stochastic growth models in demography, see for instance \cite{Adam2016} and the
references therein; for branching processes, Galton-Watson processes with immigration (or
population dependent) also exhibit a fat tail behaviour of the time returns of the chain,
see \cite{Pakes1971, Zubkov1972, KLEBANER1993115}.

In the \textit{regenerative} case (\textit{i.e.} when the chain $X$ possesses an
\textit{accessible atom}, a Harris set on which the transition probability is constant),
the distribution of the regenerative time, the return time to the atom, is a discrete
generalized Pareto \eqref{eq:Zipf} and the parameter $\beta$ is its tail index. Due to
the non-standard behaviour of traditional estimators in this context, statistical
inference for null-recurrent Markov chains is very poorly documented in the literature
\citep{Tjostheim-2001,Han-2010, Myklebust-2012,Gao2013} and, to the best of our
knowledge, estimation of the key quantity $\beta$ has not received much attention. It is
also the goal of this article to extend the use of the estimator \eqref{eq:est} to the
case where the $S_i$'s are the successive durations between the consecutive regeneration
times up to time $n$. The main difficulty naturally arises from the fact that the number
$1+N_n\geq 0$ of regeneration times (and thus the number of durations) is now random, and
it is not independent of the variables $S_1, \ldots, S_{N_n}$ (in
particular, $S_1+\cdots+S_{N_n}\leq n$ by construction). We show that the estimator
remains strongly consistent in this scenario and we provide non-asymptotic upper confidence bounds
for the absolute deviations between $\widehat{{\beta }}_{N_n}\left( k \right)$ and $\beta$.
For illustration purposes, numerical
experiments have been carried out, providing empirical evidence of the relevance of the
estimation method promoted. Extension to the general case of (pseudo-regenerative)
null-recurrent chains is also discussed, the difficulties inherent in applying the
methodology originally proposed in \cite{Bertail2006} in the positive recurrent case to
mimic regenerative Nummelin extensions \citep{Nummelin1984} being explained at length.

The paper is organized as follows. A thorough analysis of the behaviour of the estimator
\eqref{eq:est} in the i.i.d. case, illustrated by numerical experiments, is first carried
out in section \ref{sec:tail-index-estimation}. Some of the results thus established (consistency and non-asymptotic bounds)
are next extended in section \ref{sec:background} to the regenerative regular Markovian
setup, when the estimator is computed based on a single finite-length trajectory of the
atomic chain. Experimental results are also displayed and the main barrier to the
extension of the methodology promoted to general (\textit{i.e.} pseudo-regenerative)
regular null-recurrent chains is also discussed therein. Technical proofs are deferred to
the \hyperref[sec:proofs]{Appendix section}.

\section{Tail Index Estimation - The Discrete Heavy-Tailed i.i.d. Case}\label{sec:tail-index-estimation}

Throughout this section, \(S_1,\ldots,S_n\) are independent copies of a generic discrete
generalized Pareto r.v. $S$, \textit{i.e.} a r.v. $S$ with survivor function of type
\eqref{eq:Zipf}, where the parameter $\beta>0$ and the slowly varying function $L$ are
supposedly unknown. As a first go, we start to investigate the (asymptotic) behaviour of
the estimator \eqref{eq:est} in this basic general framework and next develop the
analysis in particular situations, \textit{i.e.} when the function $L$ has a specific
form.

\subsection{Main Results - Confidence Bounds and Limit Theorems}\label{subsec:main_results}

As explained in the Introduction section, the estimator \eqref{eq:est} can be viewed as
an empirical counterpart of the quantity
\begin{equation}\label{beta_k_eq}
    \beta(k)=\ln(p_k)-\ln(p_{k+1})=\beta+\ln\left(\frac{L(e^{k})}{L(e^{k+1})}\right),
\end{equation}
see \eqref{eq:Zipf}, which tends to $\beta$ as $k\rightarrow \infty$ by virtue of the
slow variation property of $L$. As previously emphasized, unless the function $L$ is
supposed to be asymptotically constant (\textit{i.e.} there exists $C>0$ s.t.
$L(x)\rightarrow C$ as $x\rightarrow +\infty$), the discrete generalized Pareto model
\eqref{eq:Zipf} is not a discrete version of the (continuous) approximately
$\beta$-Pareto model considered in \cite{Carpentier2015} and, consequently, the validity
framework established therein does not apply directly here. The proposition below
provides an upper confidence bound for the absolute deviations between \eqref{eq:est} and
$\beta$ (respectively, between \eqref{eq:est} and $\beta(k)$).

\begin{proposition}\label{bias_variance_thm}
    Let \(\delta\in (0,\; 1/2)\) and set \(u_n( \delta) = \ln ( {2/\delta })/n\) for all $n\geq 1$. If \(k\geq 1\) is such that \(p_{k+1}\geq 16 u_n(\delta)\), then, with probability at least \(1-2\delta\), we have:
    \begin{equation}\label{rateConvergence}
        \left|\widehat{\beta}_{n}(k)-\beta\right|\leq 6\sqrt{\frac{u_n (\delta)}{p_{k+1}}} + \left| \ln\left(\frac{L(e^{k})}{L(e^{k+1})}\right) \right|.
    \end{equation}
\end{proposition}

Refer to the Appendix section for the technical proof. The bound \eqref{rateConvergence}
reveals some sort of ``'bias-variance'' trade-off, ruled by the hyperparameter $k>0$. The
second term on the right-hand side can be viewed as the bias of the inference method,
insofar as the estimator \eqref{eq:est} can be seen as an empirical version of the
approximation \eqref{beta_k_eq}. It decays to $0$ as $k$ increases towards infinity,
while the first term, whose presence is due to the random nature of the estimator, tends
to $+\infty$. We point out that \textit{second-order slow variation conditions}
\citep{GoldieSmith87} are required to bound the (vanishing) bias term in
\eqref{rateConvergence}, as shall be explained in subsection \ref{subsec:examples}. The
following result reveals that for an appropriate choice of \(k=k_n\), the estimator
\eqref{eq:est} is strongly consistent.

\begin{theorem}[Strong consistency]\label{consistency}
    Suppose that, as $n\rightarrow +\infty$, we have \(k_n\to +\infty\) and $(\ln n)\exp(k_n\beta)/n=o(L(\exp(k_n))$. Then, we have:
    \begin{equation*}
        \widehat{\beta}_n(k_n)\to\beta \text{ almost surely, as } n\rightarrow +\infty.
    \end{equation*}
\end{theorem}
In particular, as stated below, strong consistency is guaranteed when \(k_n\) is of logarithmic order.

\begin{corollary}\label{ln_strong_consistency}
    Let \(0<A<1/\beta\). Then, we have:
    \begin{equation*}
        \widehat{\beta}_n({A\ln n})\to\beta \text{ almost surely, as } n\rightarrow +\infty.
    \end{equation*}
\end{corollary}

Now, the following results establish the asymptotic normality of the deviation between
\eqref{eq:est} and $\beta(k_n)$, when appropriately normalized.

\begin{theorem}[Asymptotic normality]\label{asymptotic_normality}
    Suppose that \(k_n\) satisfies the conditions of Theorem \ref{consistency}, then
    \begin{itemize}
        \item[(i)] Then, as $n\to +\infty$, we have the convergence in distribution:
              $$\sqrt {n p_{k_n}} \left(\widehat \beta_n ({k_n}) - \beta(k_n) \right) \Rightarrow \mathcal{N}\left(0,\; e^{\beta}-1\right).
              $$
        \item[(ii)]In addition, asymptotic normality holds true for the 'standardized' deviation:  $$\frac{{\sqrt {n\widehat p_{{k_n}}^{(n)}} \left( {{{\widehat \beta }_n}\left( {{k_n}} \right) - \beta \left( {{k_n}} \right)} \right)}}{{\sqrt {{e^{{{\widehat \beta }_n}\left( {{k_n}} \right)}} - 1} }}\Rightarrow \mathcal{N}\left( {0,\; 1} \right), \text{ as } n\to +\infty.$$
    \end{itemize}
\end{theorem}

The asymptotic normality results above can be extended to the deviation between
\eqref{eq:est} and $\beta$, provided that the bias term $\beta(k_n)-\beta$ vanishes at an
appropriate rate, as stated below.

\begin{corollary}\label{asymptotic_normality_real_value}
    Suppose that the conditions of Theorem \ref{asymptotic_normality} are fulfilled. In addition, assume that \(k_n\) is such that
    \begin{equation}\label{eq:cond1}
        \sqrt {np_{k_n}} \left( {1 - \frac{{L\left( {{e^{{k_n}}}} \right)}}{{L\left( {{e^{{k_n} + 1}}} \right)}}} \right) \to 0, \text{ as } n\to +\infty.
    \end{equation}
    \begin{itemize}
        \item[(i)]
              Then, we have the convergence in distribution
              \[\sqrt {np_{k_n}} \left( {\widehat \beta_n \left( {{k_n}} \right) - \beta } \right) \Rightarrow \mathcal{N}(0,\; e^\beta-1) \text{ as } n\rightarrow +\infty.\]
        \item[(ii)] In addition, the ``studentized'' version is asymptotically normal:
              \[\frac{{\sqrt {n\widehat p_{{k_n}}^{(n)}} \left( {{{\widehat \beta }_n}(k_n) - \beta } \right)}}{{\sqrt {{e^{{{\widehat \beta }_n}(k_n)}} - 1} }}\Rightarrow \mathcal{N}(0,1) \text{ as } n\to +\infty.\]
    \end{itemize}
\end{corollary}

Of course, the condition \eqref{eq:cond1} on $k_n$ can be difficult to check in
practice. This is a common issue in tail estimation and in the statistical analysis
of extreme values more generally. The choice of the hyperparameter $k$ somehow
governs the (asymptotic) bias-variance trade-off: the estimator \eqref{eq:est} is
expected to have a large variance when $k$ is large and to have a large bias
if $k$ is too small. As depicted in Fig. \ref{fig:estimator_a}, to choose $k$,
one may use the same approach as that originally proposed for the Hill estimator
(see \textit{e.g.} \cite{Resnick2007}), which consists of plotting the values
of \eqref{eq:est} for a range of values of $k$ and choosing $k$ in a region where
the estimator shows some stability. Section 3.2 of \cite{Carpentier2015} suggests
an adaptive algorithm for selecting \(k\) for \(\widehat{\beta}_n\), assuming the stricter condition
that \(F\) belongs to the class of second-order Pareto distributions.

We point out that, contrary to (variants of) the Hill inference method,
the integer hyperparameter \(k\) does not refer here to the number of order statistics involved in the estimator.
Instead, it determines a threshold that defines the tail probability we want to estimate. This implies
that the range of \(k\) is the set of integers between \(1\) and
\(\max_{1\leq i \leq n} \ln S_i-1\), while, for classical tail index estimators, the hyperparameter varies between \(1\) and \(n\).

\begin{figure}[ht!]
    \centering
    \begin{subfigure}[b]{0.49\textwidth}
        \centering
        \includegraphics[width=\textwidth]{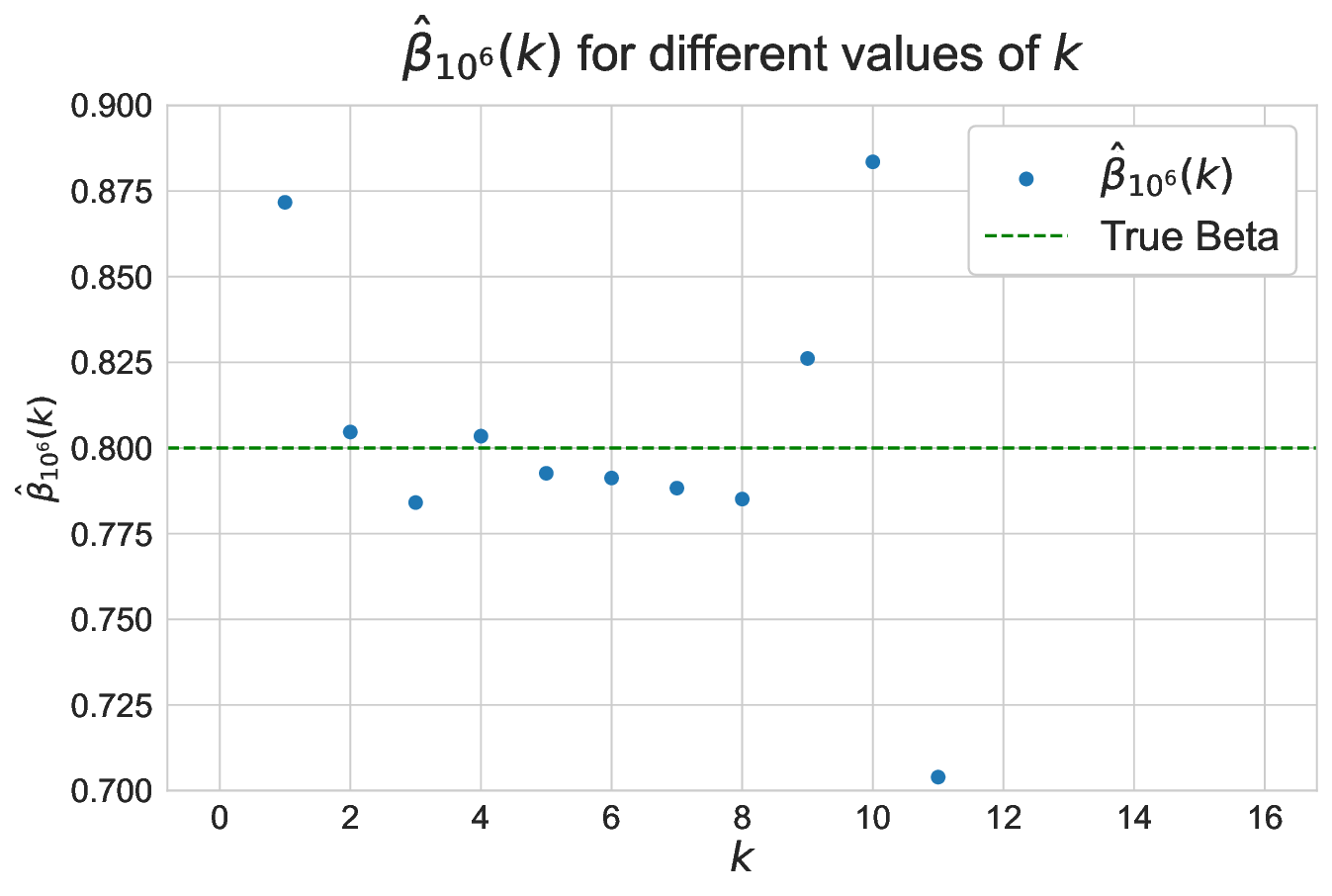}
        \caption{\(\widehat{\beta}_n(k)\) estimator}
        \label{fig:estimator_a}
    \end{subfigure}
    \hfill
    \begin{subfigure}[b]{0.49\textwidth}
        \centering
        \includegraphics[width=\textwidth]{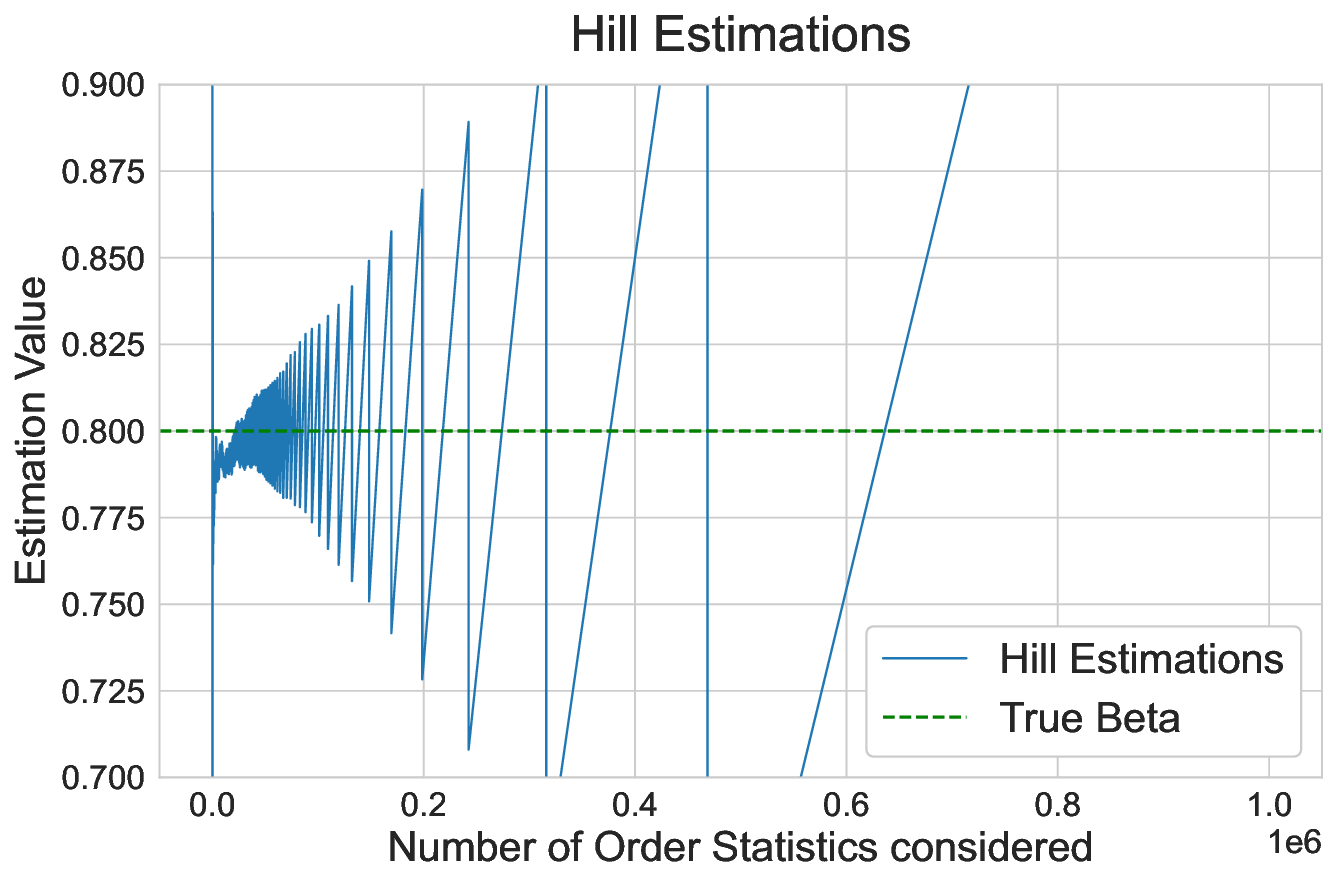}
        \caption{Hill estimator}
        \label{fig:estimator_b}
    \end{subfigure}
    \caption[Caption for LOF]{(a) Behaviour of \(\widehat{\beta}_n(k)\) for
        different values of \(k\), to estimate the parameter \(\beta=0.8\) based on a dataset of \(10^6\) independent
        realizations of a Zeta distribution\protect\footnotemark. (b) Behaviour of the Hill estimator for the same dataset.
        Notice that the range of the x axis in both graphs is different, in the case of \(\widehat{\beta}_n(k\) the range of \(k\) is the set of integers between \(1\) and \(\max_{1\leq i \leq n} \ln S_i-1\)
        while in the case of the Hill method, the x-axis is the number of order statistics
        involved in the estimator, hence, it goes from 1 to \(n\).}
    \label{fig:estimator}
\end{figure}

\footnotetext{A discrete r.v. \(W\) follows a Zeta distribution with parameter \(\beta\) if
    $\P\left( W=k \right)=1/(k^{\beta-1}\zeta( \beta-1))$ where
    \(\zeta\) is the Riemann zeta function. Its cdf satisfies $\P(
        W\geq k)\sim k^{\beta}/(\beta\zeta(\beta-1))$. This distribution is also
    known as Zipf's distribution due to its relationship with Zipf's law. }

\begin{remark}{\sc (Quantiles vs tail probabilities)} In Extreme Value Theory,
    it is customary to estimate the tail index by looking at the largest values, \textit{i.e.} a fraction of the order statistics. The general
    idea consists of fixing the tail probabilities and then using the order statistics (\textit{i.e.} empirical quantiles)
    to estimate the quantiles \cite[Chapter 3]{de2007extreme}. The estimator proposed
    here follows the converse approach: we fix the thresholds/quantiles, next use the data to
    estimate the tail probabilities and finally use these estimations in order to mimic
    the limit behavior of \eqref{beta_k_eq}. In general, estimators based on order
    statistics are difficult to analyse in a non-asymptotic way,
    whereas our estimate is based on probabilities of well-chosen tail events.
\end{remark}

\begin{remark} {\sc (Averaged versions)}
    From a practical point of view, rather than picking a single value for $k$, another natural approach would consist in averaging the estimators \eqref{eq:est} over a range of values for the hyperparameter. Let \(k\) and \(m\) be such that \(k>m\) and define
    \begin{equation*}
        \beta \left( {k,m} \right)       = \frac{1}{{2m + 1}}\sum\limits_{j =  - m}^m {\beta \left( {k + j} \right)},\;\;
        \widehat \beta_n \left( {k,m} \right)  = \frac{1}{{2m + 1}}\sum\limits_{j =  - m}^m {\widehat \beta_n \left( {k + j} \right)}.
    \end{equation*}

    One may easily check that
    \begin{equation}\label{bias_average_eq}
        \beta \left( {k,m} \right) = \beta  + \frac{1}{{2m + 1}}\left| {\ln \left( {\frac{{L\left( {{e^{k - m}}} \right)}}{{L\left( {{e^{k + m + 1}}} \right)}}} \right)} \right|.
    \end{equation}

    The non-asymptotic bound in Proposition \ref{bias_variance_thm} can be extended to
    the averaged version, as revealed by the analysis carried out in \ref{subsec:ave} in
    the Appendix section, as well as the strong consistency and asymptotic normality
    results. However, the asymptotic variance of the averaged version is shown to
    increase with $m$.
\end{remark}
\begin{remark}{\sc (Second-order conditions)}
    According to \cite[Remark 3.2.6]{de2007extreme}, under the following second order condition:
    \begin{equation}\label{eq:hill_estimator_unbiased}
        \forall t>0, \alpha>0,\quad \lim_{x\to +\infty}\left( {\frac{\overline{F}(tx)}{\overline{F}(x)}}-t^\beta \right)x^{\alpha}=0
    \end{equation}
    the Hill estimator, properly standardized, converges in distribution to a centered normal
    random variable. We point out that if the distribution of \(S\) satisfies
    \eqref{eq:hill_estimator_unbiased}, the sequence \(k_n\) satisfies the hypothesis of
    Theorem \ref{asymptotic_normality}, and there is a positive constant \(C\) such that \(n\leq e^{C k_n}\)
    for \(n\) large enough, then the conditions of Corollary \ref{asymptotic_normality_real_value} are satisfied.
\end{remark}
As explained in Example \ref{ex:matsui}, it is possible to build a discrete heavy tail distribution that do not satisfy any second order condition (as highlighted in \cite{Matsui2013}) but for which asymptotic normality of \eqref{eq:est} remarkably holds true.

\begin{example}\label{ex:matsui}

    Let \(l\in\N\), \(\beta>0\) and \(U\) be a uniform distribution in \([0,1]\).
    Consider the random variable \(X=10^{-l} \lfloor 10^l U^{-1/\beta} \rfloor \). This
    transformation of $U^{-1/\beta}$ turns all but the first $l$ digits behind the comma
    into zeros. Equation (1.6) to \cite{Matsui2013} shows that the survival function of
    this random variable is given by
    \begin{equation}\label{eq:survival_fn_matsui}
        \bar{F}(x)=\frac{(10^l)^\beta}{\left( \left\lfloor 10^l x \right\rfloor + 1 \right)^\beta},
    \end{equation}
    and that \(\bar{F}(x)\sim x^{-\beta}\) as \(x\to\infty\). Moreover, the authors showed
    that this distribution does not satisfy the second order condition \eqref{eq:hill_estimator_unbiased}.
    Define \(L(x)={(10^lx)}^{\beta}({\lfloor 10^l x \rfloor+1})^{-\beta}\), then
    \(L(x)\to 1\) as \(x\to\infty\), hence it is slowly varying, and equation
    \eqref{eq:survival_fn_matsui} can be written as \(\bar{F}(x)=x^{-\beta}L(x)\). In
    section \ref{sec:details_example_matsui} of the Supplementary Material we show that
    there exists a positive constant \(K\) such that, for \(x\) big enough
    \begin{equation}\label{eq:svf_bound_matsui}
        \left|\frac{L(x)}{L(ex)}-1\right|\leq K x^{-1},
    \end{equation}
    therefore, if \(k_n\to +\infty\), for \(n\) large enough, we have
    \begin{equation}\label{eq:bound_asympt_norm_condition}
        \sqrt {np_{k_n}} \left| {1 - \frac{{L\left( {{e^{{k_n}}}} \right)}}{{L\left( {{e^{{k_n} + 1}}} \right)}}} \right|  \leq K\sqrt{\frac{nL(e^{k_n})}{e^{k_n\beta}}}\left(e^{-k_n}\right)\leq 2K\sqrt{n e^{-k_n(\beta+2)}}.
    \end{equation}

    This shows that if \(k_n\) satisfies the conditions of Theorem \ref{consistency} and
    \(n e^{-k_n(\beta+2)}\) goes to \(0\), then the conditions of Corollary
    \ref{asymptotic_normality_real_value} are satisfied and \(\sqrt {np_{k_n}} (
    {\widehat \beta_n ( {{k_n}} ) - \beta } )\) converges in distribution to a centered
    normal random variable. Moreover, if we take \(k_n=A\ln n\) (as in Corollary
    \ref{ln_strong_consistency}) then, the asymptotic normality is guaranteed as long as
    \(1/(\beta+2)<A<1/\beta\).
\end{example}


In the next subsection, we discuss further how the behaviour of the slowly varying
function $L$ impacts the 'bias-variance' contributions revealed by the bound
\eqref{rateConvergence}.

\subsection{Refined 'Bias \textit{vs} Variance' Analysis - Examples}\label{subsec:examples}

We now consider several specific cases of distributions of type \eqref{eq:Zipf}
(\textit{i.e.} several instances of the slowly varying functions $L$) to explicit the
asymptotic order of magnitude of the terms $1/\sqrt{np_{k+1}}$ and $|
    \ln(L(e^{k})/L(e^{k+1}))|$ involved in the bound \eqref{rateConvergence}, when \(k_n\) is
picked as in Corollary \ref{ln_strong_consistency}: \(k_n=A\ln n\) with \(0<A<1/\beta\).

\begin{itemize}
    \item \textbf{The logarithmic case:} Suppose that \(L(n)=C\ln{n}\), where $C>0$. In this
          situation, we have $| \ln(L(e^{k_n})/L(e^{k_n+1}))|\sim 1/(A\ln n)$ as $n\to+\infty$,
          whereas $1/\sqrt{np_{k+1}}=O(1/\sqrt{n^{1-A\beta}\ln n})$.
    \item \textbf{The inversely logarithmic case:} Consider now the situation where
          $L(n)=C/\ln{n}$ with $C>0$. Then, we still have we have
          $|\ln(L(e^{k_n})/L(e^{k_n+1}))|\sim 1/(A\ln n)$, while $1/\sqrt{np_{k+1}}=O(\sqrt{(\ln
              n)/n^{1-A\beta}})$ as $n\to+\infty$.
\end{itemize}

We point out that, in the two examples above, the conditions of Corollary
\ref{asymptotic_normality_real_value} are not met, the bias being too big to get
asymptotic normality (centered at \(\beta\)).

\begin{itemize}
    \item \textbf{The asymptotically constant case:} Suppose that \(L(n)=e^{C_0}(1+\epsilon{(n)})\) where \(C_0>0\) and \(\epsilon(n)\to 0\) as $n\to +\infty$. In this case,
          $| \ln(L(e^{k_n})/L(e^{k_n+1}))|= O(\epsilon(n^A))$ and \(1/\sqrt{np_{k+1}}=O(1/ \sqrt{n^{1-A\beta}})\).
          Hence, if \(| \epsilon( n^A )|= O( n^{-\lambda})\)
          for some \(\lambda>0\), then the conditions of Corollary \ref{asymptotic_normality_real_value} are satisfied if we take \(k_n=A\ln n\) such that
          \(\max\{ (1-2\lambda)/\beta,\; 0 \}< A < 1/\beta\).
    \item \textbf{Slow variation with a remainder (\(SR2\)):} Consider the case where the slowly varying
          function satisfies the condition \(SR2\) introduced in \cite{Bingham1987}: there exist two real-valued functions \(k\) and \(g\) defined on $\mathbb{R}_+$ such that, for all $\lambda>0$,
          \begin{equation}\label{eq:SR2}
              \frac{L(\lambda x)}{L(x)}-1 \sim \kappa(\lambda)g(x), \text{ as } x\to+\infty,
          \end{equation}
          where \(\kappa (\lambda)=c\int_1^{\lambda}{{\theta}^{\rho-1}d\theta}\), $c>0$ and \(g\) is regularly varying with index \(\rho\leq 0\), \textit{i.e.} \(g(x)=x^\rho U(x)\) where \(U\) is a slowly varying function.
          Under the additional assumption
          that \(g\) has positive decrease, Corollary 3.12.3 in \cite{Bingham1987} gives the following representation:
          \begin{equation}\label{sr2_repre}
              L(x)=C\left( 1-c|\rho|^{-1}g(x)+o\left(g \left(x\right) \right) \right), \text{ as } x\to +\infty,
          \end{equation}
          where \(C\) is a finite constant. The result below provides precise
          control of the bias of the estimation method in this case.

          \begin{lemma}\label{bias_sr2} Suppose that conditions \eqref{eq:SR2} and \eqref{sr2_repre} are fulfilled. Then, as $n\to +\infty$, we have:
              \begin{equation*}
                  \ln\left(\frac{L\left(n^A\right)}{L\left(en^A\right)}\right)=-c|\rho|^{-1}n^{-A|\rho|}\left( U\left(n^A\right)-e^{-|\rho|}U\left( en^A \right) \right) + o\left(n^{-A|\rho|}U\left(n^A\right)\right).
              \end{equation*}
          \end{lemma}

          In this situation, the bias of the method is thus of order \(O(n^{-A|\rho|})\),
          while $1/\sqrt{np_{k+1}}$ is of order \(O(n^{-(1-A\beta)/2})\). Hence, if
          \(1/(\beta+2|\rho|)\leq A<1/\beta\), the conditions of Corollary
          \ref{asymptotic_normality_real_value} are satisfied with \(k_n=A\ln n\).
\end{itemize}

To illustrate this trade-off, we present the following Monte-Carlo experiment: We
generate $10^4$ samples of a heavy-tailed distribution and calculate
$\widehat{\beta}_{10^4}({k})$ for all admissible values of $k$, we repeat this experiment
100 times, and then we calculate the mean and the 95\% confidence interval of
$\widehat{\beta}_{10^4}({k})$ for each value of $k$. The results of these simulations,
for the cases where $L(n)$ is asymptotically constant and $L(n)$ is logarithmic, are
presented in Figures \ref{fig:constant_heavy_tail} and \ref{fig:log_heavy_tail}. As
expected, the behaviour of the estimator is way better in the former case than in the
latter.

\begin{figure}[ht!]
    \centering
    \begin{subfigure}[b]{0.49\textwidth}
        \centering
        \includegraphics[width=\textwidth]{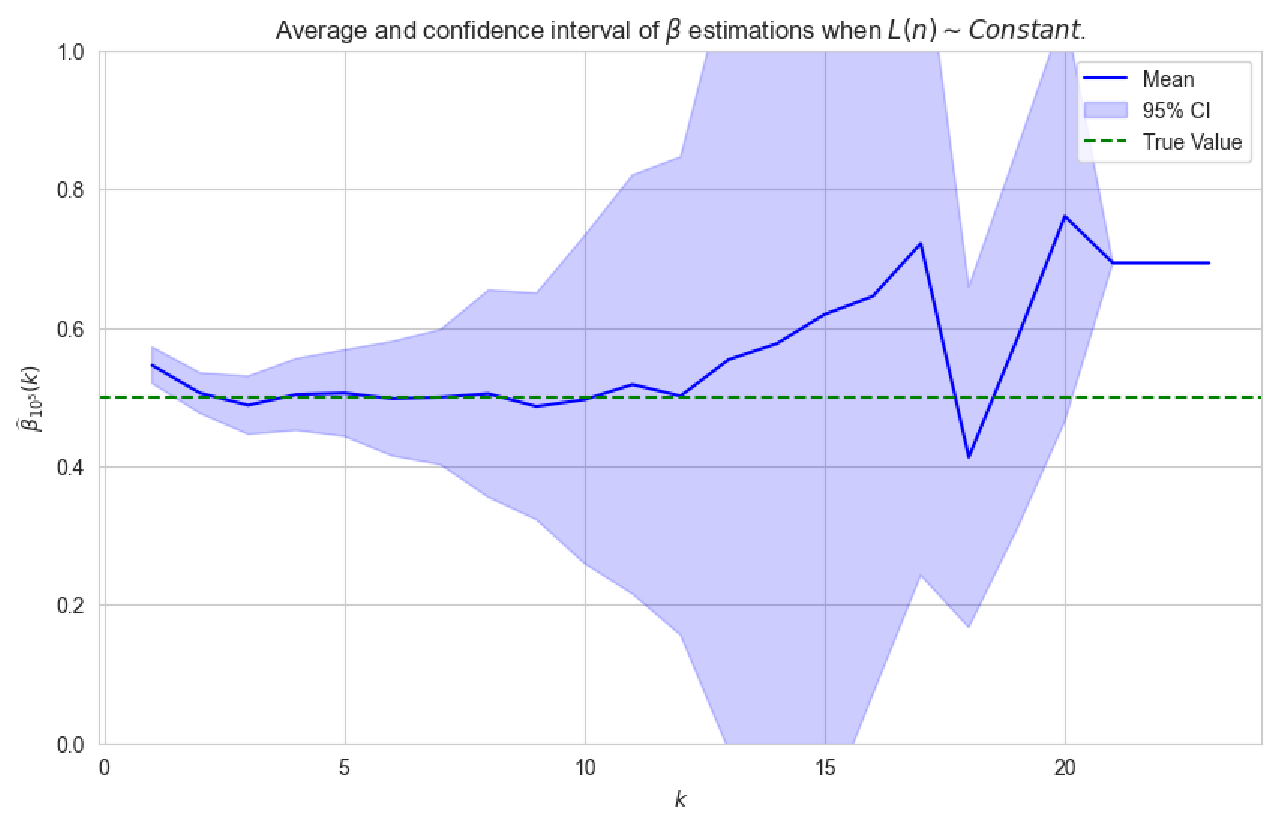}
        \caption{Case when $L(n)\sim C$}
        \label{fig:constant_heavy_tail}
    \end{subfigure}
    \hfill 
    \begin{subfigure}[b]{0.49\textwidth}
        \centering
        \includegraphics[width=\textwidth]{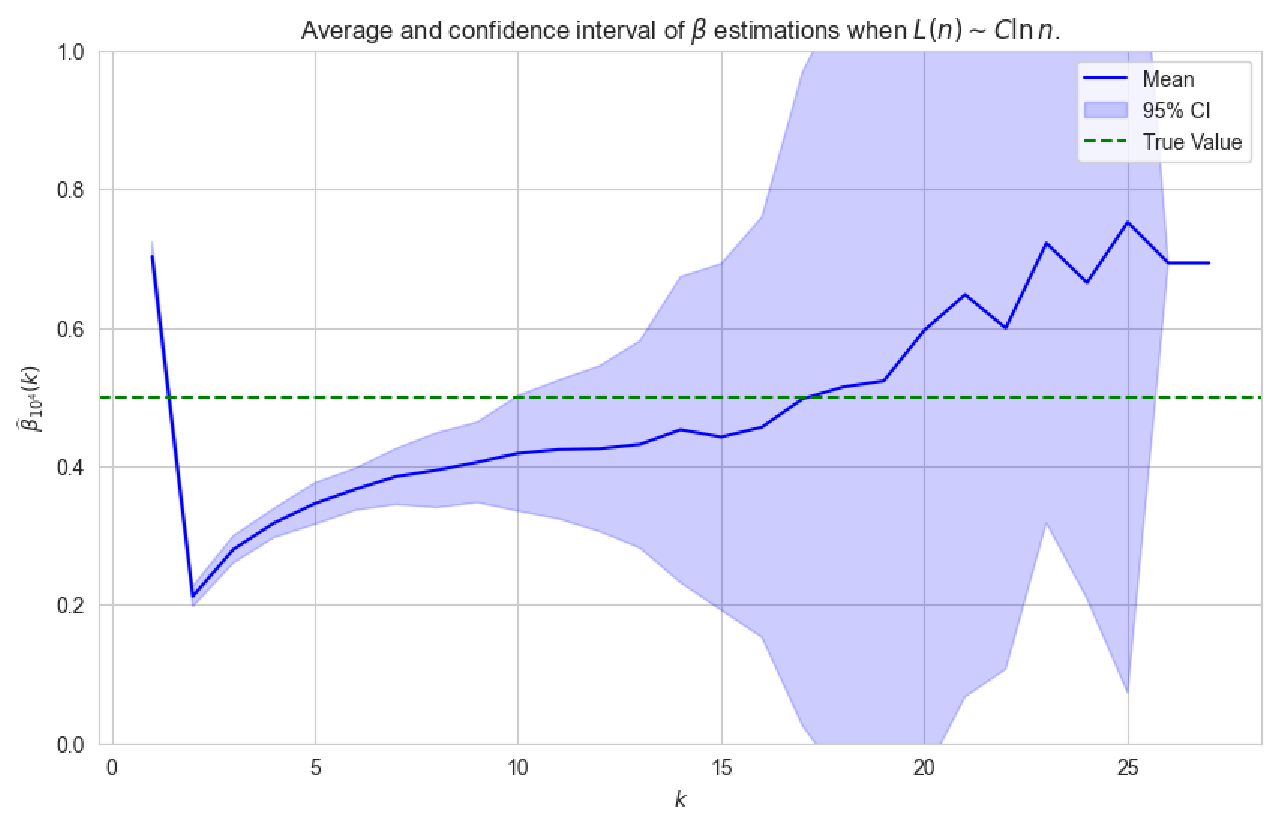}
        \caption{Case when $L(n)\sim C\ln (n)$}
        \label{fig:log_heavy_tail}
    \end{subfigure}
    \caption{Monte-Carlo average and 95\% confidence interval for \(\widehat{\beta}_{10^4}(k)\), as \(k\) is varying. The true value of $\beta$ in both cases is $0.5$.}
    \label{fig:heavy_tail_comparison}
\end{figure}

\section{Regular Null-Recurrent Chains - Regularity Index Estimation}\label{sec:background}

We start by setting out the notations used throughout this section, now standard in the
Markov chain literature, and listing first the properties supposedly satisfied by the
class of Markov chains under study. One may refer to \cite{Meyn2009} for an excellent
account of the Markov chain theory. The concept of \(\beta\)-regularity for describing
how fast a Harris chain returns to Harris sets is then recalled, together with related
asymptotic properties, invoked in the subsequent statistical analysis, for clarity's
sake. Then, the main results of this paper, related to the inference of the parameter
$\beta$ and the extended use of the estimator \eqref{eq:est} in the (regenerative)
Markovian case, are established and discussed. Here, $X=(X_n)_{n\in \mathbb{N}}$ denotes
a time-homogeneous Markov chain, with state space $E$, equipped with a countably
generated $\sigma$-field $\mathcal{E}$, and transition probability $\Pi(x,dy)$. For any
probability distribution $\nu$ on $E$, we denote by $\mathbb{P}_{\nu}$ the probability
distribution on the underlying space such that $X_{0}\sim \nu(dx)$ and by
$\mathbb{E}_{\nu}[.]$ the corresponding expectation. For notational convenience, we shall
write $\mathbb{P}_{x}$ and $\mathbb{E}_{x}[.]$ when $\nu$ is the Dirac mass at $x\in E$.
In the following, we also denote by $t\in \mathbb{R}\mapsto \lfloor t \rfloor$ the floor
function and by $\Gamma(z)=\int_{t\geq 0}t^{z-1}e^{-t}dt$ the Gamma function.

\subsection{Background and Preliminaries}\label{subsec:back}

Throughout the section, we suppose that the chain $X$ is $\psi$-irreducible, meaning that
there exists some $\sigma$-finite measure $\psi$ on $(E,\mathcal{E})$ such that any
measurable set $B\subset E$, weighted by $\psi$, can be reached by the chain with
positive probability in a finite number of steps, \textit{i.e.} $\sum_{n\geq
        1}\Pi_{n}(x,B)>0$, no matter the starting point $x\in E$, denoting by $\Pi_n(x,dy)$ the
$n$-th iterate of the transition probability $\Pi(x,dy)$. Recall that an irreducibility
measure is said to be \textit{maximal} if it dominates any other irreducibility measure.
We also assume that $X$ is aperiodic (rather than replacing $\Pi$ by an iterate) and
Harris recurrent, \textit{i.e.} that, with probability one, it visits an infinite number
of times any measurable subset $B\subset E$, weighted by maximal irreducibility measures,
whatever the initial state: $\forall x\in E$,
$\mathbb{P}_x\left(\sum_{n=1}^{\infty}\mathbb{I}\{ X_n\in B \}=\infty\right)=1$. When
Harris recurrent, a transition kernel $\Pi(x,dy)$ has a non zero invariant (positive)
measure $\mu(dx)$ (\textit{i.e.} such that $\int_{x\in E}\mu(dx)\Pi(x,dy)=\mu(dy)$), that
is unique up to a multiplicative factor (notice incidentally that $\mu(dx)$ is a maximal
irreducibility measure). Measurable sets weighted by $\mu$ are said to be Harris. For
Harris recurrent chains, recall that the following strong ratio limit theorem holds. We
have indeed, as $n\rightarrow \infty$,
\begin{equation}\label{eq:SLLN}
    \frac{\sum_{i=1}^n\mathbb{I}\{ X_i\in B \}}{\sum_{i=1}^n\mathbb{I}\{ X_i\in C \}}
    \rightarrow \frac{\mu(B)}{\mu(C)}\quad \mathbb{P}_{\nu}\text{-almost-surely},
\end{equation}
for any initial distribution $\nu$ and any measurable sets $B$ and $C$ s.t. $\mu(C)>0$.
When the measure $\mu(dx)$ is finite, the chain is said to be \textit{positive recurrent}
and, by convention, rather than considering $\mu(dx)/\mu(E)$, by $\mu(dx)$ we mean the
stationary probability measure in this case.\medskip

\noindent \textbf{Regular chains.} For a wide class of Harris Markov chains, the
\textit{regularity index} describes how fast the \textit{occupation time} related to a
Harris set $B$ (\textit{i.e.} the number of visits to $B$)
$$\Sigma_n(B)=\sum_{i=1}^n\mathbb{I}\{ X_i\in B \}$$ increases with time $n$. When $X$ is
positive recurrent, it follows from the Strong Law of Large Numbers that occupation times
of Harris sets grow in a linear fashion with the observation time: as $n\rightarrow
    \infty$, $\Sigma_n(B)\sim \mu(B)n \;\; \mathbb{P}_{\nu}\text{-almost surely}$. In the
general Harris case, some technical assumptions are required in order to be able to
specify the growing rate. In order to formulate them rigorously, further concepts are
required. Recall that a \textit{special set} (also referred to as a \textit{$D$-set}
sometimes \citep{Chen1999}) for the chain $X$ is any Harris set $D$ such that
$\mu(D)<\infty$ and $   \sup_{x\in E}\mathbb{E}_x[ \sum_{i=1}^{\tau_B}\mathbb{I}\{ X_i\in
    D \} ]<\infty$, for any Harris set $B\subset E$, denoting by $\tau_B=\inf\{i\geq 1:\;
    X_i\in B \}$ the hitting time to $B$. We recall that special sets not only exist but
there are many of them: actually, any Harris set contains a special set at least, see
Proposition 5.13 in \cite{Nummelin1984}. For any special set $D$ and initial distribution
$\nu$, consider the so-termed \textit{truncated Green function}:
\begin{equation*}
    G_{\nu,D}(t)=\frac{1}{\mu(D)}\sum_{n=1}^{\lfloor t\rfloor}\nu \Pi_n(D)
\end{equation*}
where $\nu \Pi_n(B)=\int_{x\in E}\nu(dx)\Pi_n(x,B)=\mathbb{P}_{\nu}(X_n\in B)$ for any
$B\in \mathcal{E}$. Harris recurrence entails that $G_{\nu, D}(t)\rightarrow +\infty$ as $t\rightarrow
    \infty$.

In the following, we restrict our attention to a specific class of Harris chains for
which the rate at which $G_{\nu, D}(t)$ grows to infinity as $t\rightarrow \infty$ can be
characterized. Notice that, in such cases, the rate would be independent from the pair
$(\nu, D)$. Indeed, by virtue of Theorem 7.3 in \cite{Nummelin1984}, we have $G_{\nu_1,
            D_1}(t)/G_{\nu_2, D_2}(t)\rightarrow 1$ as $t$ goes to infinity, for any distributions
$\nu_1$ and $\nu_2$ and any special sets $D_1$ and $D_2$. One may thus give the following
definition, see \cite{Chen1999,Chen2000}.

\begin{definition}[{\sc $\beta$-regular Markov chain}]\label{def:regular_chain} Let $\beta\in [0,1]$.
    A Harris chain $X$ is said to be $\beta$-regular if there exists a special
    set $D$ and a distribution $\nu$ such that the function $G_{\nu, D}$ is
    $\beta$-regularly varying: $\forall t>0$,
    \begin{equation}\label{eq:regular_chain}
        \lim_{\lambda\rightarrow \infty}\frac{G_{\nu, D}(\lambda t)}{G_{\nu, D}(\lambda)}=t^{\beta}.
    \end{equation}
\end{definition}

We point out that property \eqref{eq:regular_chain} can be rephrased as follows: there
exists a slowly varying function $L_{\nu, D}(t)$ such that $G_{\nu, D}(t)=L_{\nu,
    D}(t)t^{\beta}$. Notice incidentally that ``$\beta$-regularity'' is called ``$\beta$-null
recurrence'' in \cite{Tjostheim-2001} when $\beta<1$, while $\beta=1$ corresponds to the
positive recurrent case. The parameter $\beta$ thus rules the ``frequency'' at which a
(supposedly regular) Harris chain $X$ recurs \citep{Chen1999} and it is the purpose of
this section to investigate the issue of estimating it with asymptotic guarantees, based
on the observation of a single path $X_1,\; \ldots,\; X_n$ of size $n\to+\infty$. In
particular, we shall focus in subsection \ref{subsec:main} on the case of
\textit{regenerative} chains, for which an extension of the estimator \eqref{eq:est} can
be used with statistical guarantees. \medskip

\noindent \textbf{Regenerative regular chains.} Recall that a Markov chain is
\textit{regenerative} when it possesses an accessible atom, \textit{i.e.} a measurable
set $A$ such that $\psi(A)>0$ and $\Pi(x,.)=\Pi(y,.)$ for all $(x,y)\in A^{2}$. By
$\tau_{A}=\tau_{A}(1)=\inf\left\{ n\geq1,\;X_{n}\in A\right\}$ is meant the hitting time
to $A$ and we denote by $\tau_{A} (j)=\inf\left\{ n>\tau_{A}(j-1),\;X_{n}\in A\right\} ,
$ for $j\geq2$, the successive return times to $A$, by $\mathbb{P}_{A}$ the probability
measure on the underlying space such that $X_{0}\in A$ and by $\mathbb{E}_{A}[.]$ the
$\mathbb{P}_{A}$-expectation. In the regenerative case, it results from the
\textit{strong Markov property} that the blocks of observations in between consecutive
visits to the atom
\begin{equation}\label{eq:traj}
    \mathcal{B}_{1}=(X_{\tau_{A}(1)+1},...,\;X_{\tau_{A}(2)}),\; \ldots,\;
    \mathcal{B}_{j}=(X_{\tau_{A}(j)+1},...,\;X_{\tau_{A}(j+1)}),\;\ldots
\end{equation}
form a collection of i.i.d. random variables, taking their values in the torus
$\mathbb{T}=\cup_{n=1}^{\infty}E^{n}$, and the sequence $\{\tau_{A}%
    (j)\}_{j\geq1}$, corresponding to successive times at which the chain forgets
its past is a (possibly delayed) renewal process.

The class of regenerative chains includes a wide variety of Markov processes, including
all chains with countable state spaces (where any recurrent state serves as an accessible
atom). This class also incorporates many Markov models frequently employed in operations
research and queueing theory (see, for example, \cite{Asmussen2010}). Further examples of
(regular) regenerative chains are presented in Examples \ref{ex1} and \ref{ex3} below.

\begin{example}\label{ex1}{\sc (Bessel random walks)} A \textit{Bessel random walk} with drift $\delta\in [-1,\; +\infty)$ is a Markov
    chain with $\mathbb{N}$ as state space, jumps in $\{-1,\; +1\}$, reflecting at $0$
    and with transition probabilities of the form:
    \begin{equation*}
        \Pi(0,1) = 1 \quad \text{and} \quad 1-\Pi(k, k-1) = \Pi(k, k+1) = \frac{1+h(k)-\delta/(2k)}{2}\quad \forall k \geq 1,
    \end{equation*}
    where
    $h(k)\in(-1+\delta/(2k),\; 1+\delta/(2k))$ and \(h( k) = o(1/k)\) as \(k\to+\infty\).
    It is recurrent when \(\delta>-1\), positive recurrent when \(\delta>1\) and
    transient when \(\delta=-1\). For \(\delta=1\), it is either recurrent or else
    transient, depending on the function \(h(x)\). In the null recurrent case, the chain
    is \(\beta\)-regular with \(\beta = (1+\delta)/2\) , see Theorem 2.1 in
    \cite{Alexander2011}. Of course, when \(\delta=0\) and \(h\equiv 0\), this chain
    corresponds to a simple reflected random walk with \(p=1/2\).
\end{example}


\begin{example}\label{ex3}{\sc (Null recurrent, not necessarily regular, chains)}
    By means of the model below, originally presented in \cite{Myklebust-2012}, one can
    generate \(\beta\)-null recurrent chains for any \(\beta>0\), as well as null
    recurrent chains that are not regular. Let \(\{\eta_n\}_{n\in \mathbb{N}}\) be a
    sequence of i.i.d. real-valued random variables. Consider the chain defined by:
    \begin{equation*}
        X_n=(X_{n-1}-1)\mathbb{I}\{X_{n-1}>1\}+\eta_n \mathbb{I}\{
        X_{n-1}\in [0,1]\}\quad n\geq 1.
    \end{equation*}
    This chain is regenerative, with the interval \(\left[ {0,1} \right]\) as atom.
    In addition, we have \(\mathbb{P}_x( {{\tau _{\left[ {0,1} \right]}} > n}) =
    \mathbb{P}( {\lfloor {\eta _1}} \rfloor > n )\). Hence, \(X\) is null recurrent iff
    \(\mathbb{E}{\left\lfloor {{\eta _1}} \right\rfloor }=\infty\) and \(\beta\)-regular
    with $\beta\in (0,1]$ iff the r.v. $\lfloor \eta _1 \rfloor$ has generalized discrete
    Pareto distribution with tail index $\beta$.
\end{example}

In the regenerative setting, all stochastic stability properties may be expressed in
terms of speed of return to the atom. For instance, when $X$ is \textit{Harris
    recurrent}, see Theorem 10.0.1 in \cite{Meyn2009}, the invariant measure is equal to the
occupation measure between two consecutive visits to the atom (up to a multiplicative
factor): $\forall B\in \mathcal{E}$, $\mu(B)\sim \E_{A}\left[ \sum
    _{i=1}^{\tau_{A}}\mathbb{I}\{X_{i}\in B\}\right]$. For instance, the chain is positive
recurrent if and only if the expected return time to the atom is finite\footnote{Its
(unique) invariant probability distribution $\mu$ is then given by
$\mu(B)=(1/\mathbb{E}_{A}[\tau_{A}])\mathbb{E}_{A}[
    \sum_{i=1}^{\tau_{A}}\mathbb{I}\{X_{i}\in B\}],\text{ for all } B\in\mathcal{E}$.},
\textit{i.e.} $\mathbb{E}_{A}[\tau_{A}]<\infty$, see Theorem 10.2.2 in \cite{Meyn2009}.
More generally, the $\beta$-regularity property can be characterized by the heaviness of
the tail of the probability distribution of the regeneration times in the atomic case, as
the following result shows.%

\begin{proposition}\label{prop:atomic} (\cite{Tjostheim-2001}, Theorem 3.1) Suppose that
    $X$ is regenerative Harris recurrent. Let $A$ be an atom for $X$ and
    $\beta\in [0,1]$. The following assertions are equivalent.
    \begin{itemize}
        \item[(i)] The chain $X$ is $\beta$-regular.
        \item[(ii)] There exists a slowly varying function $L_A:\mathbb{R}_+\rightarrow \mathbb{R}_+$ such that: $\forall m\geq 1$,
              \begin{equation}
                  \mathbb{P}_A\left( \tau_A \geq m \right)=L_A(m)\cdot m^{-\beta}.
              \end{equation}
    \end{itemize}
\end{proposition}

As a direct consequence of Proposition \ref{prop:atomic}, we have that if $X$ is
regenerative and $\beta$-regular, then $\,\beta=\sup\{\theta\in [0,1]:\;\;
    \mathbb{E}_A\left[\tau_A^{\theta} \right]<+\infty\}$.

Based on the decomposition \eqref{eq:traj} of the whole sample path, limit theorems for
regenerative Markov chains can be derived from the application of their i.i.d.
counterparts to the sequence of blocks $(\mathcal{B}_k)_{k\geq 1}$, see \textit{e.g.}
\cite{Meyn2009}. This approach is usually referred to as the \textit{regenerative method}
and is extensively used to establish the asymptotic results stated in subsection
\ref{subsec:limit} in the atomic case. Notice however that the regenerative blocks
$\mathcal{B}_{1},\; \ldots,\; \mathcal{B}_{N_n}$, where $\Sigma_n(A)=N_n-1$ denotes the
(random) number of regenerations before time $n$, forming the truncated trajectory up to
time $n$ are not independent (the sum of their length being less than $n$ in particular),
which causes technical difficulties when establishing higher-order or non-asymptotic
results, see \textit{e.g.} \cite{Bertail2006} or \cite{bertail2004edgeworth} and the
references therein.

The inference technique for the regularity index $\beta$ of the chain $X$ developed in
subsection \ref{subsec:main} is based on characterization $(ii)$: the parameter $\beta$
is the tail index of a discrete generalized Pareto r.v., the regeneration time namely,
\textit{i.e.} the conditional survivor function of $\tau_A$ given $X_0\in A$.
Incidentally, notice that the parameter $\beta$ does not depend on the atom $A$
considered (in contrast to the estimator analysed in subsection \ref{subsec:main}).

Based on a (random) number $N_n$ of (dependent) realizations of the regenerative time, namely
\begin{equation*}
    S_j=\tau_A(j+1)-\tau_A(j) \text{ for } j=1,\; \ldots,\; N_n,
\end{equation*}
one may naturally compute the estimator \eqref{eq:est}. As will be shown in Theorem \ref{th:consistency_markov}, in spite of
the dependence structure between the $S_j$'s, the consistency property is preserved in
the Markovian framework.


\subsection{Limit Theorems of the occupation times for Regular Markov Chains}\label{subsec:limit}

We now recall the limit results related to the behaviour of the random occupation times
$\Sigma_n(.)$ for regular Markov chains and discuss their limitations regarding their
possible use to infer the regularity index $\beta$ with (asymptotic) guarantees. The
latter essentially reveals that the empirical occupation measures $\Sigma_n(B)$ of Harris
sets $B$ grow at the sublinear rate $n^{\beta}$ (up to a slowly varying factor). As shall
be seen, however, due to the great dispersion of their (asymptotic) distribution, the
empirical occupation measures can hardly be used directly to estimate the key parameter
$\beta$.

The result stated below claims that the logarithm of the occupation time of any Harris
set provides a strongly consistent estimator of the regularity index $\beta$ when
appropriately normalized. It corresponds to the comment in Remark 3.7 of
\cite{Tjostheim-2001}.

\begin{proposition}
    Suppose that the chain $X$ is $\beta$-regular with $\beta\in (0,1)$ and $B$ is a Harris set. Let $\nu$ be its initial probability distribution. Then, we have:
    \begin{equation}\label{eq:as_convergence_occupation_time_estimator}
        \widetilde{\beta}_n(B)\to \beta\;\; \mathbb{P}_{\nu}- \text{a.s., as } n\to +\infty,
    \end{equation}
    where
    \begin{equation}\label{eq:est3}
        \widetilde{\beta}_n(B)=\frac{\ln \Sigma_n(B)}{\ln n}.
    \end{equation}

\end{proposition}

It was pointed out in Remark 3.7 of \cite{Tjostheim-2001} that this estimator is of
limited practical use due to its slow rate of convergence, although no specific rate was
given therein. Equation (4.2) in \cite{Chen1999} may suggest that
$|\widetilde{\beta}_n(B)-\beta|$ is almost-surely $O(\ln{K_{\nu ,D}(n)}/\ln n)$ as $n\to
    +\infty$. However, this has not been proven. In the following result, we obtain the limit
distribution of the estimator and show that under the additional assumption that
$L_{\nu,D}$ has a finite non-zero limit, the estimator has a logarithmic rate of
convergence.

\begin{theorem}\label{th:rate_convergence}Suppose that the chain $X$ is $\beta$-regular with $\beta\in (0,1)$, $B$ is a Harris set, and $\nu$ is its initial probability distribution. Then, as $n\to +\infty$,
    \begin{equation*}
        \ln(n)\left( \widetilde{\beta}_n(B)-\beta - {\frac{\ln{L_{\nu,D}}}{\ln{n}}} \right)\Rightarrow \ln{\left( \mu(B)/(Z_{\beta})^{\beta} \right)} \text{ in } \mathbb{P}_{\nu}\text{-distribution}.
    \end{equation*}
    where $Z_{\beta}$ is a stable random variable with Laplace transform
    \[\psi_{\beta}\left ( t \right )=\exp\left (-t^{\beta}/\Gamma( \beta+1) \right ),\quad t\geq 0.\]
    In addition, if $\lim_{n\to +\infty}L_{\nu,D}(n)$ exists and is not $0$, then, there
    exists a constant $\kappa>0$ such that
    \begin{equation*}
        \ln(n)\left( \widetilde{\beta}_n(B)-\beta \right)\Rightarrow \ln{\left( \kappa/(Z_{\beta})^{\beta} \right)} \text{ in } \P_{\nu}\text{-distribution}.
    \end{equation*}
\end{theorem}


The almost sure convergence of $\widetilde{\beta}_n(B)$ towards $\beta$ suggests that,
for $n$ large enough, $\ln\Sigma_n(B)\approx\beta\ln(n)$ and that the log-log plot of
$\Sigma_n(B)$ and $n$ should look like a linear function with slope $\beta$, which could
be possibly used to infer the value of $\beta$. Unfortunately, the dispersion of such a
plot (and that of the process $\sigma_n(B)$, asymptotically described by Theorem
\ref{thm:funclimit} of the Supplementary Material) is way too large in practice. To
illustrate this, we simulated a Simple Symmetric Random Walk ($\beta=0.5$) with $n=10^5$
points, and we computed $\ln\Sigma_{\lfloor {nt} \rfloor}(B)$ for $0.1\leq t\leq 1$
(choosing $B=\{0\}$, which is an atom for this regenerative regular chain). The outcomes
of this simulation are depicted in Fig. \ref{fig:log_log_plot}.

\begin{figure}[ht!]
    \centering
    \includegraphics[width=0.65\textwidth]{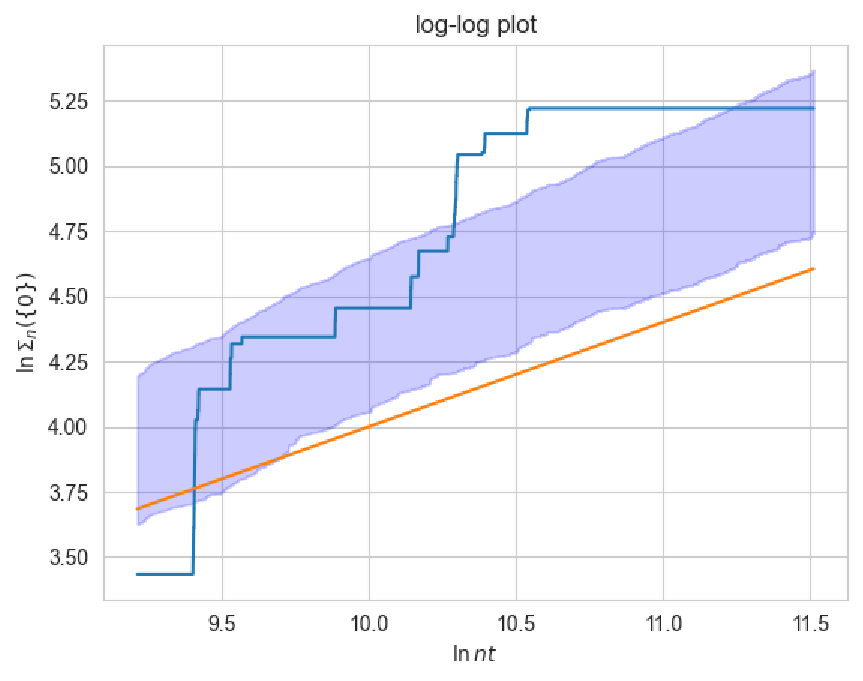}
    \caption{Log-log plot of $\ln\Sigma_{\lfloor {nt} \rfloor}(\{0\})$ in dark blue, the orange line representing the plot of the linear function $x\mapsto 0.5\times x$, while the blue area represents the 95\% confidence interval for $\ln\Sigma_{\lfloor {nt} \rfloor}(\{0\})$ calculated from 100 independent trajectories.}
    \label{fig:log_log_plot}
\end{figure}

This simulation illustrates in particular the slow convergence of
$\widetilde{\beta}_n(B)$ described in Corollary \ref{th:rate_convergence}. Hence, in the
regenerative case, it is more suitable to exploit the tail behaviour of the regenerative
times (\textit{cf} Proposition \ref{prop:atomic}) to estimate the regularity index
$\beta$, as shall be investigated in the next subsection.

\subsection{Regularity Index of a Regular Chain - Statistical Inference}\label{subsec:main}

Assume that $X$ is a regenerative regular chain with atom $A$ and unknown regularity
index $\beta\in (0,1)$, and suppose that a sample path $X_1,\; \ldots,\; n$ of length
$n\geq 1$ is observed. Because the chain is Harris recurrent, the number of observed
regeneration times $\Sigma_n(A)$ almost-surely tends to $+\infty$ as $n\to+\infty$.
Hence, with probability 1, we have $N_n\geq 1$ for $n$ large enough and one can define
\begin{equation}\label{eq:emp_surv}
    \widehat{p}^{(N_n)}_l=\frac{1}{N_n}\sum_{i=1}^{N_n}\mathbb{I}\{S_i>e^l \} \text{ for any } l\in \mathbb{R}
\end{equation}
and form the statistic
\begin{equation}\label{eq:est2}
    \widehat{\beta}_{N_n}(k)=\ln \left( \widehat{p}^{(N_n)}_k/\widehat{p}^{(N_n)}_{k+1} \right),
\end{equation}
provided that $\widehat{p}^{(N_n)}_{k+1}>0$. We point out that, due to the randomness of
$N_n$, \eqref{eq:emp_surv} is a biased (strongly consistent) estimator of
$p_l=\mathbb{P}_A(\tau_A>e^l)$ for any $l\in \mathbb{R}$. The estimator \eqref{eq:est2}
is of the same form as \eqref{eq:est} except that the number $N_n$ of observations is
random, and it is not independent of the sequence \(\{S_i\}\) of observations
(in particular $S_1+\ldots+S_{N_n}\leq n$).

Obtaining properties of this new estimator requires understanding the behavior of
\(N_n\), to this, the key is the following result, proved in \cite{Tjostheim-2001}, which gives
the asymptotic distribution of \(N_n\).

\begin{proposition}\label{prop:markov:limit_dist}(\cite{Tjostheim-2001}, Theorem 3.2) If \(X\) is an atomic \(\beta\)-regular
    Markov chain with initial probability distribution \(\nu\), then
    \begin{equation*}
        N_n \overline{F}(n) \Rightarrow \frac{M_{\beta}(1)}{\Gamma(1-\beta)} \text{ in } \P_{\nu}\text{-distribution}.
    \end{equation*}
    where \(M_{\beta}(1)\) is a Mittag-Leffler distribution with parameter \(\beta\).
\end{proposition}

From Proposition \ref{prop:markov:limit_dist} we can deduce the following corollary which
provides a deterministic control of \(N_n\) with high probability.

\begin{corollary}\label{cor:markov:deterministic_control} For any \(\delta \in (0,1)\), let \(H_\delta> R_\delta>0\) be such that
    \begin{equation*}
        \P \left(\frac{M_\beta(1)}{\Gamma(1-\beta)}\in \left[R_\delta, H_\delta\right]\right)\geq 1-\delta/2.
    \end{equation*}
    Then, there exists \(n_\delta\)
    such that
    \begin{align*}
        \forall n\geq n_\delta\quad\P_{\nu}\left( N_n\in \left[ \frac{R_{\delta}}{\overline{F}(n)}, \frac{H_{\delta}}{\overline{F}(n)} \right] \right)>1-\delta.
    \end{align*}
\end{corollary}

Equipped with Corollary \ref{cor:markov:deterministic_control}, we can provide the following result,
which is the analogue of Proposition \ref{bias_variance_thm} in the Markovian case.

\begin{proposition}\label{prop:markov:bias_variance_thm}
    Let \(X\) be an atomic Markov chain, \(\beta\)-regular with \(\beta\in(0,1)\) and with initial probability
    distribution \(\nu\). Let \(\delta\in (0,\; 1/2)\), take \(H_\delta, R_{\delta}\) and \(n_\delta\)
    as in Corollary \ref{cor:markov:deterministic_control} and set \(v_n( \delta) = \ln ( {36/\delta})\overline{F}(n)/H_{\delta}\).
    Then, as soon as \(n\geq n_\delta\) and \(p_{k+1}\geq {({40H_\delta}R^{-1}_\delta)}^2 v_n(\delta)\), we have with probability larger than $1-2\delta$:
    \begin{equation*}
        \left|\widehat{\beta}_{N_n}(k)-\beta\right|\leq \frac{60 H_\delta}{R_{\delta}}\sqrt{\frac{v_n(\delta)}{p_{k+1}}} + \left| \ln\left(\frac{L(e^{k})}{L(e^{k+1})}\right) \right|.
    \end{equation*}
\end{proposition}

\begin{theorem}\label{th:consistency_markov} Suppose that the atomic chain $X$ is $\beta$-regular
    with $\beta\in (0,1)$. Let $\nu$ be its initial probability distribution. If
    \(k_n\to +\infty\) s.t. $(\ln n)\exp(k_n\beta)/n=o(L_A(\exp k_n))$ as
    $n\rightarrow +\infty$, then
    the estimator \eqref{eq:est2} is strongly consistent:
    \begin{equation}
        \widehat{\beta}_{N_n}(k_{N_n})\to \beta \;\;\; \mathbb{P}_{\nu}-a.s \text{ as } n\to +\infty.
    \end{equation}
    In particular, strong consistency holds for $\widehat{\beta}_{N_n}(A\ln N_n)$ with \(A<1\).
\end{theorem}

\begin{remark}\label{remark:convergence_rates} {\sc (On investigating convergence rates)} Due to the impossibility of tightly
    controlling the sequence $N_n$ by a deterministic quantity in probability and the non-linearity of
    the estimator, we have not been able to extend to the Markovian case the asymptotic normality results
    of Theorem \ref{asymptotic_normality} and Corollary
    \ref{asymptotic_normality_real_value} via Anscombe's theorem \cite[Theorem 7.3.2]{Gut2013}.
    Heuristically, if in Theorem \ref{asymptotic_normality} we take
    $k_n=\ln N_n$ and replace $N_n$ by
    its approximate expectation $n^{\beta}L_{\nu, A}(n)$ \cite[Lemma 3.3]{Tjostheim-2001},
    we would get a convergence rate of order $n^{-\beta\left(1-\beta\right)/2}L_1\left(n\right)$, where
    $L_1\left(n\right)$ is the slowly varying function given by
    $\sqrt{L_{\nu, A}(n^\beta L_{\nu, A}(n))/L_{\nu, A}(n)^{1-\beta}}$.
    This suggests a convergence rate of order $n^{-\beta(1-\beta)/2}$ when $L_{\nu, A}\sim C>0$. However,
    we have not been able to prove this claim.
\end{remark}

\begin{remark} {\sc (Trajectories of random length)} Suppose that the trajectory is observed until
    the $N$-th regeneration, \textit{i.e.} $n=\tau_A(N)$, with $N\geq 2$. In this case, we will obtain a
    sequence of $N$ i.i.d blocks whose sizes follow the heavy-tailed distribution described in
    \eqref{eq:Zipf}, and therefore, the results of Section \ref{subsec:main_results} can be applied directly
    to this sequence. Notice that in this case, the total number of points observed in the chain
    (i.e. the amount of time we need to wait to collect the $N$ blocks) is a random variable that, while
    finite with probability one, has infinite expectation.
\end{remark}

\begin{remark} {\sc (The (atomic) positive recurrent case)} When the chain is positive recurrent (or equivalently
    $1$-regular), the estimator \eqref{eq:est2} can be naturally used to estimate the tail index $\beta'\geq 1$ of
    the regeneration time, when the latter has a regularly varying distribution. Dedicated theoretical results can be
    found in section \ref{subsec:positive_recurrent_case} of the Appendix.
\end{remark}

\subsection{Simulation Experiments}

In order to analyze empirically the finite sample behavior of our estimator in the Markovian case,
we dedicate this subsection to a simulation example. As a generative model for our experiments we
use the Bessel random walks, defined in Example \ref{ex1}, as it will allow us to study
the accuracy of our estimator, not just for different sample sizes, but also for different
values of \(\beta\).

For different values of \(\beta\) and \(n\) we have performed the following two-fold experiment:
\begin{enumerate}
    \item First, we have generated one path of a $\beta$-regular Bessel random walk of length \(n\). We have then computed the estimator
          \(\widehat{\beta}_{N_{n}}(k)\) based on this trajectory and plotted the results in order to determine
          the value \(k_{\beta,n}\) where it first stabilizes.
    \item We have simulated \(2000\) independent paths of a $\beta$-regular Bessel random walk of length \(n\). Then, we have computed the estimator
          $\widehat{\beta}_{N_{n}}(k_{\beta,n})$ based on each of these trajectories, as well as the estimator
          $\widetilde{\beta}_n(\{0\})$.
\end{enumerate}

The results of this experiment are presented in Table~\ref{tab:experiment_results}.
They show that the accuracy of our estimator greatly increases as the length \(n\) gets larger (notice
that the bias and variance are both divided by two in order of magnitude when \(n\) increases from
\(10^3\) to \(10^6\)). The histograms shown in Figure~\ref{fig:monte_carlo_markov_chain}
hint a possible asymptotic normality of the estimator $\widehat{\beta}_{N_{n}}$.
The simulation study also suggests that for moderately large sample sizes ($n\geq 10^5$) our estimator
$\widehat{\beta}_{N_{n}}$ outperforms \(\widetilde{\beta}_n\) as it has less bias
and comparable variance, while for smaller sample sizes \(\widetilde{\beta}_n\) works better.
This behavior can be attributed to the nature of our estimator, which focuses on
estimating the probabilities of tail events \(\{S > e^k\}\) to approximate the quantity
\(\beta(k)\) and use the fact that, as \(k\) tends to infinity, \(\beta(k)\) converges to
\(\beta\) (see \eqref{beta_k_eq}). When \(n\) is small, there
are fewer sample points available in the tail, which forces us to select smaller values
of \(k\) to obtain accurate estimations of the tail probability, which worsens the
estimation of \(\beta\).

\begin{table}[ht!]
    \centering
    \begin{tabular}{|c|c|c|cc|cc|}
        \hline
        \multirow{2}{*}{$\beta$} & \multirow{2}{*}{$n$} & \multirow{2}{*}{$k_{\beta,n}$} & \multicolumn{2}{c|}{\(|\textnormal{Bias}|\)} & \multicolumn{2}{c|}{Variance}                                                            \\
                                 &                      &                                & $\widehat{\beta}_{N_{n}}$                    & $\widetilde{\beta}_n(\{0\}$   & $\widehat{\beta}_{N_{n}}$ & $\widetilde{\beta}_n(\{0\})$ \\ \hline
        \multirow{4}{*}{0.5}     & $10^3$               & 2                              & 0.22255                                      & \textbf{0.07987}              & 0.11062                   & \textbf{0.01932}             \\
                                 & $10^4$               & 2                              & 0.12246                                      & \textbf{0.06697}              & 0.06032                   & \textbf{0.01317}             \\
                                 & $10^5$               & 3                              & \textbf{0.01024}                             & 0.05424                       & 0.02921                   & \textbf{0.00825}             \\
                                 & $10^6$               & 3                              & \textbf{0.00343}                             & 0.04622                       & 0.01309                   & \textbf{0.00591}             \\ \hline
        \multirow{4}{*}{0.7}     & $10^3$               & 2                              & 0.20290                                      & \textbf{0.13154}              & 0.12575                   & \textbf{0.01580}             \\
                                 & $10^4$               & 3                              & \textbf{0.02885}                             & 0.10598                       & 0.05673                   & \textbf{0.00969}             \\
                                 & $10^5$               & 3                              & \textbf{0.00500}                             & 0.08249                       & 0.01680                   & \textbf{0.00604}             \\
                                 & $10^6$               & 3                              & \textbf{0.00476}                             & 0.06961                       & \textbf{0.00389}          & 0.00411                      \\ \hline
        \multirow{4}{*}{0.9}     & $10^3$               & 3                              & 0.27182                                      & \textbf{0.20158}              & 0.15869                   & \textbf{0.00966}             \\
                                 & $10^4$               & 3                              & \textbf{0.03128}                             & 0.16307                       & 0.04685                   & \textbf{0.00539}             \\
                                 & $10^5$               & 4                              & \textbf{0.02338}                             & 0.13529                       & 0.01807                   & \textbf{0.00286}             \\
                                 & $10^6$               & 5                              & \textbf{0.00704}                             & 0.11502                       & 0.00606                   & \textbf{0.00177}             \\ \hline
    \end{tabular}
    \caption{Bias-Variance results of the simulation example.}\label{tab:experiment_results}
\end{table}


\begin{figure}[ht!]
    \centering
    \begin{subcaptiongroup}
        \centering
        \parbox[b]{.49\textwidth}{ \centering \includegraphics[width=0.49\textwidth,
                height=0.4\textwidth]{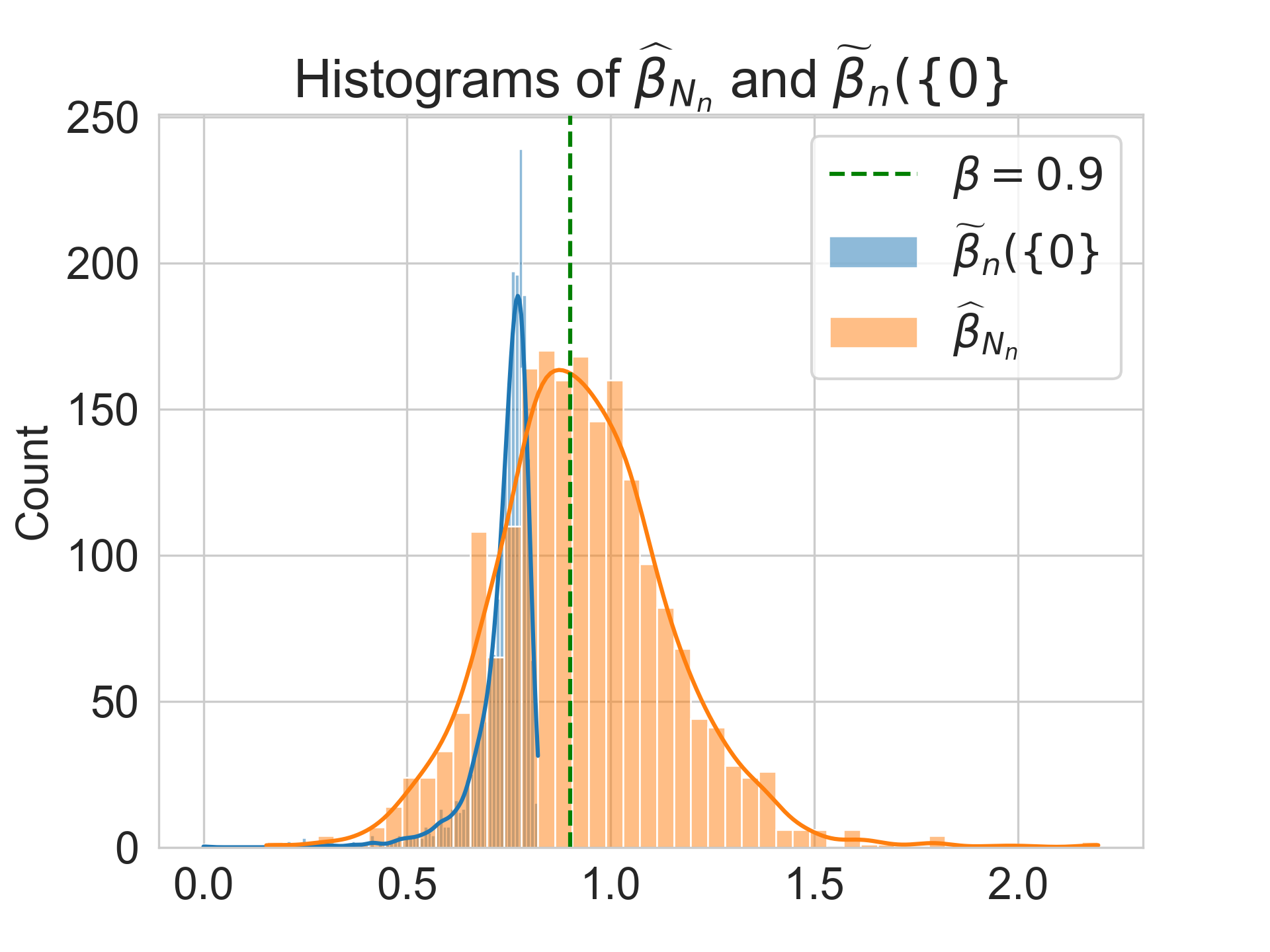} \caption{$\beta=0.9$, $n=10^4$}}
        \parbox[b]{.49\textwidth}{ \centering \includegraphics[width=0.49\textwidth,
                height=0.4\textwidth]{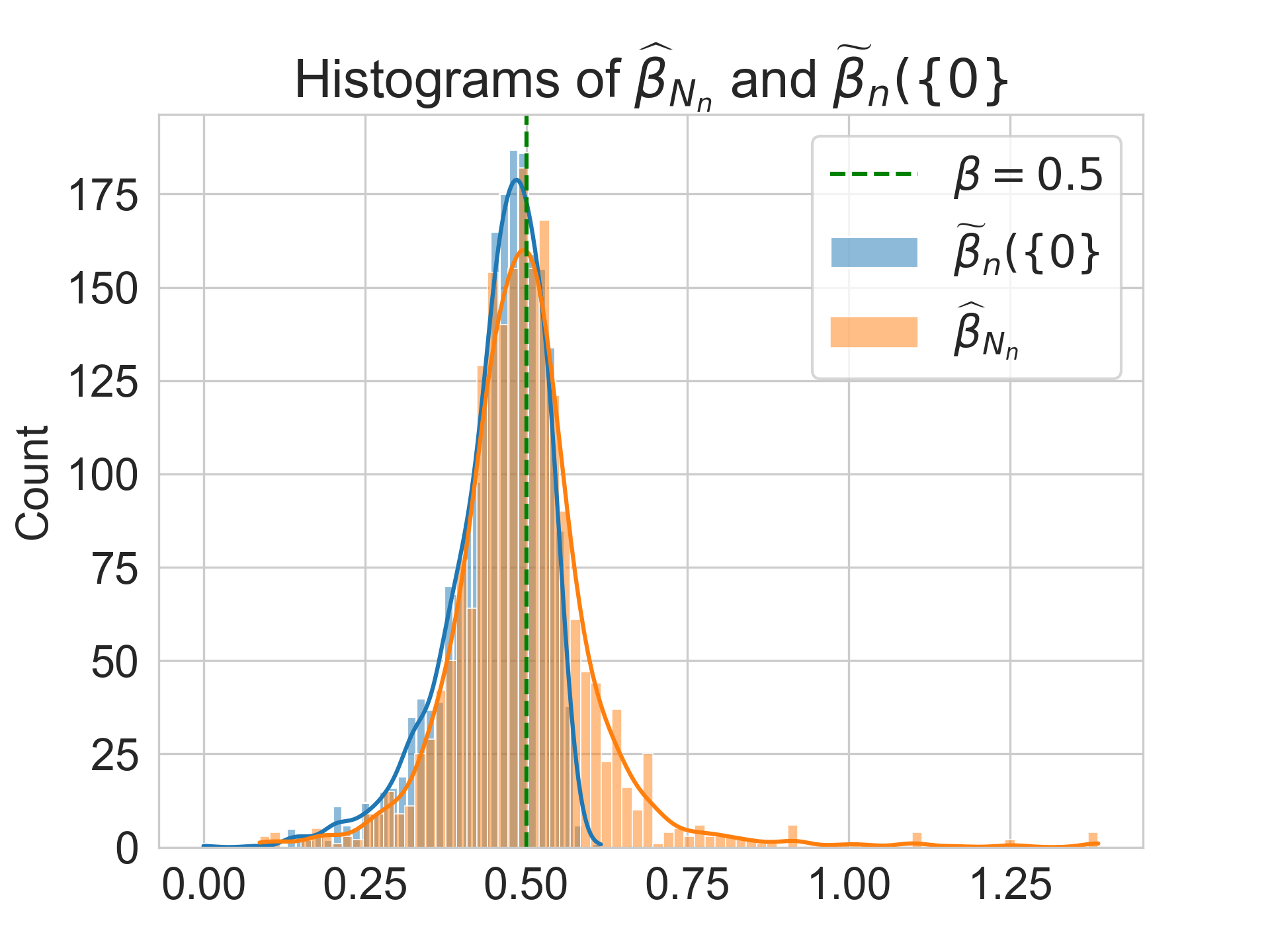}
            \caption{$\beta=0.5$, $n=10^6$}}
    \end{subcaptiongroup}
    \caption{KDE for the estimators $\widehat{\beta}_{N_{n}}$ and $\widetilde{\beta}_{N_{n}}(\{0\})$.}
    \label{fig:monte_carlo_markov_chain}
\end{figure}

\subsection{Perspectives - Extension to the Pseudo-regenerative Case}\label{subsec:extension_nummelin}

Harris chains are not necessarily regenerative, of course. However, the construction
proposed in \cite{Nummelin1978, Nummelin1984}, referred to as the \textit{Nummelin
    splitting technique}, permits to build a regenerative extension of any Harris chain (see
section \ref{sec:pseudo_regeneration} for a detailed description). Essentially, given a
Harris recurrent Markov chain \(X\), this technique allows the construction an atomic
chain, called the ``split chain'' that has a very similar communication structure as the
original chain. Moreover, if \(X\) is \(\beta\)-regular, then the ``split chain'' is also
\(\beta\)-regular with the same value of \(\beta\) \cite[pp. 19]{Chen1999}.

A data-driven algorithm (see section \ref{sec:samples_split_chain}) for obtaining samples
of the split chain, given a finite sample of \(X\) and the transition kernel density
\(\pi\), has been proposed in \cite{BertailClemencon2006}. We applied this algorithm to a
simple random walk on $\R$, defined as $X_0=0$, $X_{n+1}=X_{n}+Z_n$ for $n\geq 1$, where
$\{Z_n\}$ is a sequence of independent standard normal random variables.\footnote{The
    transition kernel density of this Markov chain is $\pi(x,y)=f(y-x)$, where $f$ is the
    standard normal density function.} After constructing the split chain, we applied our
estimator to it. The results, presented in Fig. \ref{fig:random_walk_construction}, show
that when using the estimator $\widehat{\beta}_{N_n}(k)$ on the split chain, the
estimates of $\beta$ closely approximate the true value ($\beta=1/2$) when $k$ is chosen
near $\ln N_{n}$.

\begin{figure}[ht!]
    \centering
    \begin{subcaptiongroup}
        \centering
        \parbox[b]{.49\textwidth}{ \centering
            \includegraphics[width=0.49\textwidth]{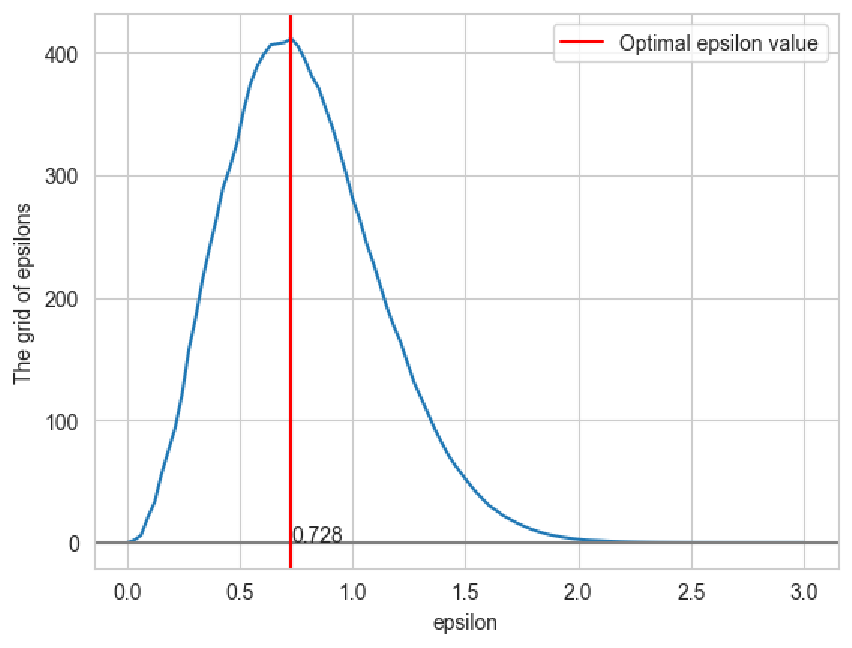}
            \caption{Selection of the value of $\epsilon$.}}
        \parbox[b]{.49\textwidth}{ \centering
            \includegraphics[width=0.49\textwidth]{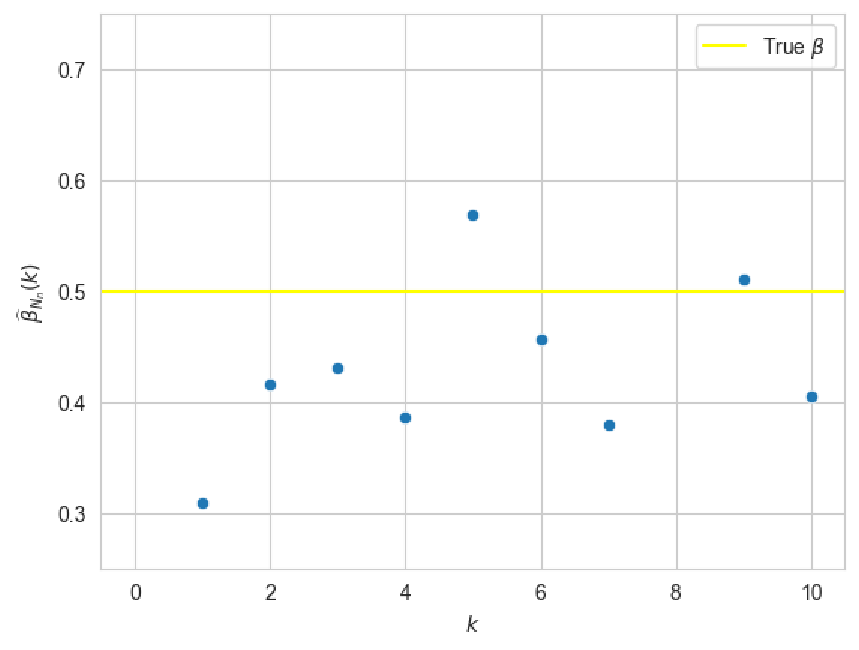}
            \caption{$\beta$ estimations.}}
        \parbox[b]{.65\textwidth}{ \centering
            \includegraphics[width=0.6\textwidth]{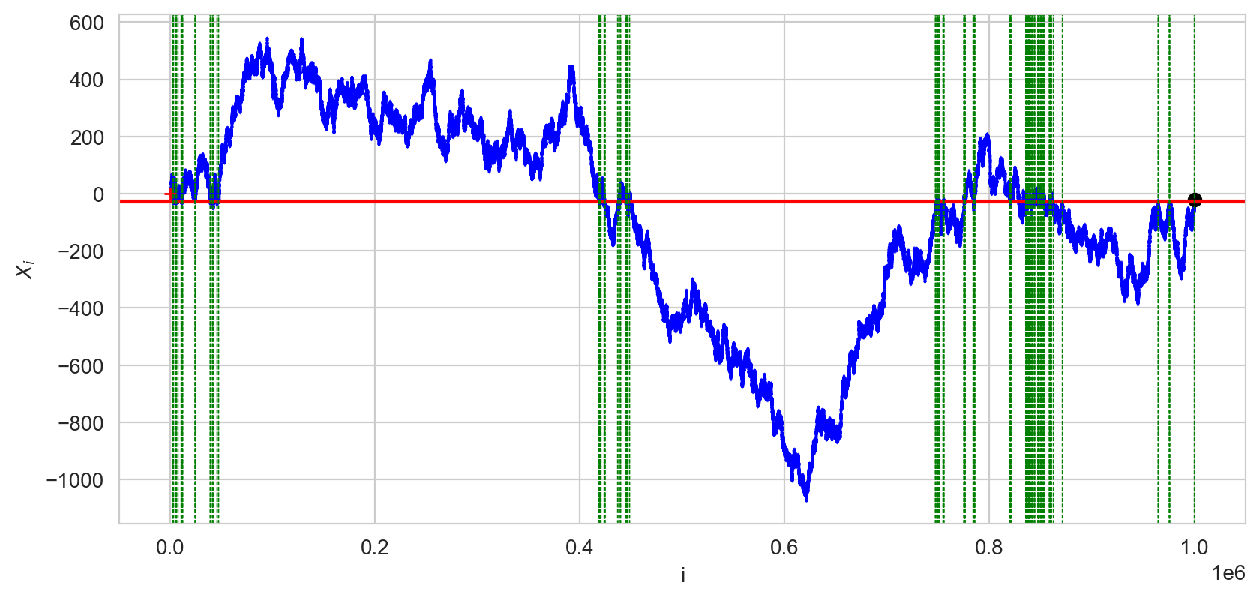}
            \caption{ Dividing the trajectories into data blocks corresponding to the
                pseudo-regeneration times. The small set (-28.2203, -26.7643) is marked in
                red, the visits to the pseudo-atom are shown as green dotted lines.}}
    \end{subcaptiongroup}
    \caption{Application of the pseudo-regeneration technique to construct the split chain and estimate
    $\beta$ using $\widehat{\beta}_{N_n}(k)$ in the random walk $X_{n+1}=X_{n}+Z_n$ where $Z_n$
    is an i.i.d. sequence of standard normal random variables. The total number
    of observed points in the chain is $10^6$ and the number of pseudo-blocks ($N_n$) is 418.}
    \label{fig:random_walk_construction}
\end{figure}

When the kernel density is unknown, a procedure to approximate the split chain, based on
an estimation $\widehat{\pi}_n$ of the kernel was presented as Algorithm 3 in
\cite{BertailClemencon2006}. As indicated in Theorem 3.1 of \cite{Bertail2006}, the
accuracy of this construction depends on the rate at which $\pi$ is estimated by
$\widehat{\pi}_n$. To our knowledge, in the null-recurrent case, the sole consistent
estimator of the transition density documented in the literature is the Nadaraya-Watson
estimator. The proof of the consistency can be found in section 5 of
\cite{Tjostheim-2001}. However, no results regarding its rate of convergence have been
established so far. Moreover, the practical choice of the bandwidth parameter involved in
this estimator is a difficult and largely unresolved problem, as discussed on pp. 412 in
\cite{Tjostheim-2001}. Hence, an ambitious line of further research consists in
understanding how to implement practically the approximate regenerative block
construction presented in \cite{BertailClemencon2006} in order to extend the estimation
methodology studied in the previous subsection to the regular pseudo-regenerative case
with theoretical guarantees.

\backmatter








\section*{Declarations}

\begin{itemize}
    \item Funding: This research has been conducted as part of the project Labex MME-DII
          (ANR11-LBX-0023-01).
    \item Competing Interests: The authors have no competing interests to declare that
          are relevant to the content of this article.
    \item Ethics approval: Not applicable
    \item Consent to participate: Not applicable
    \item Consent for publication: Not applicable
    \item Availability of data and materials: Not applicable
    \item Code availability: The code used to generate the results in the paper is
          available at \url{https://github.com/carlosds731/tail_index_estimation}
    \item Authors' contributions: This paper is part of Carlos Fernández PhD thesis.
\end{itemize}


\bibliography{biblio}

\newpage

\setcounter{page}{1}

\begin{appendices}
\section{Technical Proofs}\label{sec:proofs}

This appendix contains the technical proofs of all the new results presented in the
paper. Throughout this section we will use $u_n(\delta)$ to denxote $\ln \left(
    {2/\delta } \right)/n$ for $\delta>0$ and $n\in\N$.

\subsection{Proof of Proposition \ref{bias_variance_thm}}

By the triangular inequality and equation \eqref{beta_k_eq}, we have
\begin{equation*}
    \left|\widehat{\beta}_n(k)-\beta\right| \leq
    \left|\widehat{\beta}_n(k)-\beta(k)\right| + \left| \ln\left(\frac{L(e^{k})}{L(e^{k+1})}\right) \right|.
\end{equation*}
Proposition \ref{bias_variance_thm} now follows by Lemma \ref{varianceBoundLemma}.

\begin{lemma}\label{varianceBoundLemma}
    Let \(\delta>0\) and \(k\) such that \(p_{k+1}\geq 16u_n(\delta)\), then
    \begin{equation}\label{varianceBoundEq}
        \left|\widehat{\beta}_n(k)-\beta(k)\right|\leq 6\sqrt{\frac{u_n (\delta)}{p_{k+1}}},
    \end{equation}
    with probability larger than \(1-2\delta\).
\end{lemma}
\begin{proof}
    In order to prove this result, we need the following lemma, proved in the
    Supplementary Material of \cite{Carpentier2015}.

    \begin{lemma}\label{bernstein}\textbf{Bernstein's inequality for Bernoulli random variables}
        Let \(W_1,\ldots,W_n\) be i.i.d. samples from a distribution \(F\), and we
        define \(p_k=1-F(e^k)\), \({\widehat p^{(n)}_k} =n^{-1}\sum_{i = 1}^n {\I\{ {{W_i} > {e^k}} \}}\).
        Let \(\delta>0\)
        and take \(n\) large enough so that \(p_k\geq 4u_n(\delta)\). Then, with
        probability \(1-\delta\),
        \begin{equation*}
            \left|\widehat{p}^{(n)}_k-p_k\right|\leq 2\sqrt{p_k u_n(\delta)}.
        \end{equation*}
    \end{lemma}
    Because \(p_k\geq 16 u_n(\delta)\) we can apply the previous lemma, then
    with probability greater than \(1-\delta\) we have
    \begin{align*}
        -2\sqrt{p_k u_n(\delta)}                          & \leq \widehat{p}^{(n)}_k-p_k\leq 2\sqrt{p_k u_n(\delta)}                         \\
        p_k\left(1-2\sqrt{\frac{u_n(\delta)}{p_k}}\right) & \leq \widehat{p}^{(n)}_k \leq p_k\left(1+2\sqrt{\frac{u_n(\delta)}{p_k}}\right),
    \end{align*}
    taking the log in the previous equation we get
    \begin{align}
        \ln{p_k}+\ln{\left(1-2\sqrt{\frac{u_n(\delta)}{p_k}}\right)} & \leq \ln{\widehat{p}^{(n)}_k} \leq \ln{p_k}+\ln{\left(1+2\sqrt{\frac{u_n(\delta)}{p_k}}\right)}\nonumber   \\
        \ln{\left(1-2\sqrt{\frac{u_n(\delta)}{p_k}}\right)}          & \leq \ln{\widehat{p}^{(n)}_k} - \ln{p_k} \leq \ln{\left(1+2\sqrt{\frac{u_n(\delta)}{p_k}}\right)}\nonumber \\
        -3\sqrt{\frac{u_n(\delta)}{p_k}}                             & \leq \ln{\widehat{p}^{(n)}_k} - \ln{p_k} \leq 2\sqrt{\frac{u_n(\delta)}{p_k}},\label{ineque}
    \end{align}
    where the last pair of inequalities is obtained by using \(\ln{(1+x)}\leq
    x\) and \(\ln{(1-x)}\geq-3x/2\) for \(x<1/2\). Inequality \eqref{ineque} implies that
    \begin{equation}
        \left|\ln{\widehat{p}^{(n)}_k} - \ln{p_k}\right| \leq 3\sqrt{\frac{u_n(\delta)}{p_k}}\label{ineque2}
    \end{equation}
    with probability bigger than 1-\(\delta\). Applying \eqref{ineque2} for \(k+1\) we get with probability bigger than
    1-\(\delta\) that
    \begin{equation}
        \left|\ln{\widehat{p}^{(n)}_{k+1}} - \ln{p_{k+1}}\right| \leq 3\sqrt{\frac{u_n(\delta)}{p_{k+1}}}.\label{ineque3}
    \end{equation}
    Combining the triangular inequality with \eqref{ineque2} and \eqref{ineque3}
    completes the proof.
\end{proof}

\subsection{Proof of Theorem \ref{consistency}}

The first element in the proof of Theorem \ref{consistency} is the following simple
lemma, that shows that the non-empirical version of $\widehat{\beta}_n(k)$ converges
to $\beta$.

\begin{lemma}\label{limit_alpha}
    \begin{equation*}
        \lim_{k\to +\infty}{\beta(k)}=\beta.
    \end{equation*}
\end{lemma}
\begin{proof}
    Because \(L\) is slowly varying, \(\lim_{x\to\infty}{L(t x)}/{L(x)}=1\)
    (see 1.2.1 of \cite{Bingham1987}) for all \(t>0\), therefore
    \({L(e^{k})}/{L(e^{k+1})}={L(e^{k})}/{L(ee^k)}\to 1\) and the result
    follows after taking limits in (\ref{beta_k_eq}).
\end{proof}

Let \(\epsilon>0\). Because \(k_n\to\infty\), Lemma \ref{limit_alpha} implies that
\(\beta(k_n)\to\beta\), therefore we can find \(N_1\in\N\) such that, for all $n\geq
    N_1$
\begin{equation}\label{epsilon_n1}
    \left|\beta(k_n)-\beta\right|\leq \frac{\epsilon}{2}.
\end{equation}

Take \(\delta = {2}/{n^2}\), then \(u_n(\delta)={2\ln{n}}/{n}\). Because \(L\) is
slowly varying, \(L\left(e^{k_n+1}\right)\sim L\left(e^{k_n}\right)\), then
\(e^{k_n\beta}{\ln{n}}/{n}=o\left(L\left(e^{k_n+1}\right)\right)\) and we can find
\(N_2\in\N\) such that for all for $n\geq N_2$ we have
\(p_{k_n+1}={L(e^{k_n+1})}/{e^{(k_n+1)\beta}}\geq {32\ln{n}}/{n}=16u_n(\delta)\).
Therefore, we can apply Lemma \ref{varianceBoundLemma}, obtaining that, for all
$n\geq N_2$,
\begin{equation}\label{eq_n2}
    \P\left(\left|\widehat{\beta}_n(k_n)-\beta(k_n)\right|\leq 6\sqrt{\frac{2\ln{n}}{np_{k_n+1}}}\right)\geq 1-\frac{4}{n^2}.
\end{equation}

Combining the triangular inequality with equations (\ref{epsilon_n1}) and
(\ref{eq_n2}) we have that for all \(n\geq\max\left(N_1,N_2\right)\)
\begin{equation}\label{eq_bound}
    \left|\widehat{\beta}_n{(k_n)}-\beta\right| \leq 6\sqrt{\frac{2\ln{n}}{np_{k_n+1}}}+\frac{\epsilon}{2}
\end{equation}
with probability bigger than \(1-{4}/{n^2}\). Plugging
\(p_{k_n+1}={L(e^{k_n+1})}/{e^{(k_n+1)\beta}}\) in the first term of the
right-hand side of (\ref{eq_bound}), we get
\begin{equation*}
    6\sqrt{\frac{2\ln{n}}{np_{k_n+1}}}=6\sqrt{\frac{2\ln{n}}{n}\frac{e^{(k_n+1)\beta}}{L(e^{k_n+1})}}=6\sqrt{2e^{\beta}}\sqrt{\frac{\ln{n}}{n}\frac{e^{(k_n)\beta}}{L(e^{k_n+1})}}.
\end{equation*}

The assumption that \(e^{k_n\beta}{\ln{n}}/{n}=o(L(e^{k_n+1}))\) implies that the
above equality converges to \(0\), therefore we can find \(N_3\in\N\) such that
\(|6\sqrt{{2\ln{n}}/{(np_{k_n+1}}}|\leq {\epsilon}/{2}\) for all \(n\geq N_3\). Then,
for all \(n\geq\max(N_1,N_2, N_3)\)
\begin{equation*}
    \P\left(\left|\widehat{\beta}_n(k_n)-\beta(k_n)\right|\leq\epsilon\right)\geq 1-\frac{4}{n^2},
\end{equation*}
which shows that \(\widehat\beta_n(k_n)\) converges in probability to \(\beta\).
Moreover, because the series \(\sum_n {4}/{n^2}\) converges, Borell-Cantelli lemma implies
that \(\widehat{\beta}_n(k_n)\to\beta\) almost surely.

\subsection{Proof of Corollary \ref{ln_strong_consistency}}

If we take \(k_n=A\ln{n}\), we have \(e^{\beta k_n}=n^{A\beta}\) then
\begin{equation*}
    \lim_n{e^{k_n\beta}\frac{\ln{n}}{nL(e^{k_n})}} =\lim_n{\frac{\ln{n}}{n^{(1-A\beta)/2}}\frac{1}{n^{(1-A\beta)/2}L(n^A)}}=0.
\end{equation*}

For the last limit we have used that if \(L\) is slowly varying, then \(L(n^A)\) is
also slowly varying and that \(\lim_n n^\gamma L(n)\to +\infty\) for \(\gamma>0\)
\cite[Proposition 1.3.6.v, pp. 16]{Bingham1987}. Corollary
\ref{ln_strong_consistency} now follows by Theorem \ref{consistency}.

\subsection{Proof of Theorem \ref{asymptotic_normality} and Corollary \ref{asymptotic_normality_real_value}}

Before starting with the proof of the asymptotic normality of the
$\widehat{\beta}_n\left({k_n}\right)$ estimator, we need the following technical
lemmas.

\begin{lemma}\label{log_transform_weakly_conv}
    Let \(W_n\) be a sequence of positive random variables and
    \(a_n\) and \(b_n\) two positive sequences such that \(a_n>0\),
    \({b_n}/{a_n}\to 0\). If there exists a random variable \(W\)
    with continuous distribution function \(F\) such that ${({W_n} - {a_n})}/{b_n}$ converges in distribution to $W$,
    then ${a_n} {{{(\ln {W_n} - \ln {a_n})}}/{{{b_n}}}}$ also
    converges in distribution to $W$.
\end{lemma}
\begin{proof}
    Let \(x\in\mathbb{R}\) be fixed. Because \({({W_n} - {a_n})}/{b_n}\Rightarrow  W\) in distribution,
    we have
    \begin{equation*}
        \P\left( {{W_n} \leqslant {a_n} + {b_n}x} \right) \to F\left( x \right).
    \end{equation*}

    Using that \({a_n} + {b_n}x = {a_n}\left( {1 + a_n^{-1}b_nx} \right)\) and taking
    logs we get
    \begin{equation*}
        \P\left( {\ln {W_n} \leqslant \ln {a_n} + \ln \left( {1 + \frac{{{b_n}}}{{{a_n}}}x} \right)} \right) \to F\left( x \right).
    \end{equation*}

    The condition \({b_n}/{a_n}\to 0\) implies that $\ln \left( {1 + a_n^{-1}b_nx}
        \right) = a_n^{-1}b_nx + o\left( a_n^{-1}b_n \right)$. Then, $\P\left(
        {{a_nb_n^{-1}}(\ln {W_n} - \ln {a_n}) \leqslant x + o\left( 1 \right)} \right)$
    converges to $F\left( x \right)$ and the Lemma follows by the continuity of
    \(F\).
\end{proof}

\begin{lemma}\label{convergence_probability_p_k}
    If \(k_n\) satisfies the hypothesis of Theorem \ref{consistency}, then,
    \[\frac{{\widehat p_{{k_n}}^n}}{{\bar F\left( {{e^{{k_n}}}} \right)}}\to 1\quad \textit{almost surely}.\]
\end{lemma}

\begin{proof}
    By Lemma \ref{bernstein}, for any \(\delta>0\) such that \(p_k\geq 4
    u_n\left(\delta\right)\) we have that,
    \begin{equation}\label{concentration_prob}
        \P\left( {\left| {\frac{{\widehat p_k^n}}{{{p_k}}} - 1} \right| \leqslant 2\sqrt {\frac{{{u_n}\left( \delta  \right)}}{{{p_k}}}} } \right) \geqslant 1 - \delta.
    \end{equation}

    As in the proof of Theorem \ref{consistency}, let \(\delta = {2}/{n^2}\), so
    \(u_n(\delta)={2\ln{n}}/{n}\). The condition \(e^{k_n\beta}{\ln{n}}/{n}=o\left(
    {L\left( {{e^{{k_n}}}} \right)} \right)\) implies that we can find \(N_1\in\N\)
    such that \({p_{{k_n}}} \geqslant {{8\ln n}}/{n}\) for all \(n\geq N_1\),
    therefore, by equation \eqref{concentration_prob}, we have, for all $n \geqslant
        {N_1}$
    \begin{equation*}
        \P\left( {\left| {\frac{{\widehat p_{{k_n}}^n}}{{{p_{{k_n}}}}} - 1} \right| \leqslant 2\sqrt {\frac{{2\ln n}}{{n{p_{{k_n}}}}}} } \right) \geqslant 1 - \frac{2}{{{n^2}}}.
    \end{equation*}

    Let \(\epsilon>0\). Notice that \({{\ln n}}{{(n{p_{{k_n}}})^{-1}}} =
    {{{e^{{k_n}\beta }}\ln n }}/{(nL\left( {{e^{{k_n}}}} \right))}\) and this goes to
    0 as $n$ goes to \(+\infty\), therefore, we can find \(N_2\) such that \(2\sqrt
    {{{2\ln n}}/{{(n{p_{{k_n}}})}}} \leqslant \epsilon\) for all \(n\geq N_2\), then,
    for all $n \geqslant \max \left( {{N_1},{N_2}} \right)$
    \begin{equation*}
        \P\left( {\left| {\frac{{\widehat p_{{k_n}}^n}}{{{p_{{k_n}}}}} - 1} \right| \leqslant \epsilon } \right) \geqslant 1 - \frac{2}{{{n^2}}},
    \end{equation*}
    and the result follows by Borel-Cantelli's Lemma.
\end{proof}

The next result can be obtained using the same arguments of Example 3.11 on
\cite{Drees2010}.

\begin{lemma}\label{dress_and_rotzen} Let \(W_n\) be a sequence of i.i.d. random variables with survival function \eqref{eq:Zipf}, \(\phi_1\)
    and \(\phi_2\) bounded functions and \(u_n\) an increasing sequence of real numbers such that \(u_n\to +\infty\).
    Define
    \begin{equation*}
        {W_{n,i}} = \frac{{{W_i}}}{{{u_n}}}\mathbb{I}\left\{ {\frac{{{W_i}}}{{{u_n}}} > 1} \right\}, \quad {v_n} = \P\left( {{W_{n,i}} \ne 0} \right) \quad \textnormal{and}
    \end{equation*}
    \begin{equation*}
        {\widetilde Z_n}\left( {{\phi _k}} \right) = \frac{1}{{\sqrt {n{v_n}} }}\sum\limits_{i = 1}^n {\left( {{\phi _k}\left( {{W_{n,i}}} \right) - \E{\phi _k}\left( {{W_{n,i}}} \right)} \right)}.
    \end{equation*}

    If the following conditions hold:
    \begin{itemize}
        \item[\AssumpAOneText]\label{assumpt:AssumpAOne} \(n v_n\to +\infty\),
        \item[\AssumpATwoText]\label{assumpt:AssumpATwo} \(\E{\left[ {{{\phi _k}\left( {{W_{n,1}}} \right)}^4 } \right]} = O\left( {{v_n}} \right),\quad, k = 1,2,\)
        \item[\AssumpAThreeText]\label{assumpt:AssumpAThree} \(\mathop {\lim }\limits_n \frac{1}{{{v_n}}}{{\E\left[ {{\phi _k}\left( {{W_{n,1}}} \right){\phi _l}\left( {{W_{n,1}}} \right)} \right]} }  = {\sigma _{kl}},\quad k,l \in \{1,2\},\)
    \end{itemize}
    then \({\left( {{{\widetilde Z}_n}\left( {{\phi _k}} \right)} \right)_{1 \leqslant k \leqslant 2}}\) converges weakly to a centred normal distribution
    with covariance matrix \({\left( {{\sigma _{kl}}} \right)_{1 \leqslant k,l \leqslant 2}}\).
\end{lemma}

Let \(k_n\) satisfy the conditions of Theorem \ref{consistency}, take
\(u_n=e^{k_n}\), \({\phi _1}\left( x \right) = \mathbb{I}\left\{ {x > 1} \right\}\)
and \({\phi _2}\left( x \right) = \mathbb{I}\left\{ {x > e} \right\}\). With this
notation:
\begin{align*}
    {\phi _1}\left( {{W_{n,i}}} \right)                & = \mathbb{I}\left\{ {\left( {\frac{{{W_i}}}{{{u_n}}}\mathbb{I}\left\{ {\frac{{{W_i}}}{{{u_n}}} > 1} \right\}} \right) > 1} \right\} = \mathbb{I}\left\{ {\frac{{{W_i}}}{{{u_n}}} > 1} \right\}, \\
    {\phi _2}\left( {{W_{n,i}}} \right)                & = \mathbb{I}\left\{ {\left( {\frac{{{W_i}}}{{{u_n}}}\mathbb{I}\left\{ {\frac{{{W_i}}}{{{u_n}}} > 1} \right\}} \right) > e} \right\} = \mathbb{I}\left\{ {\frac{{{W_i}}}{{{u_n}}} > e} \right\}, \\
    \E\left[{\phi _1}\left( {{W_{n,i}}} \right)\right] & = \P\left( {\frac{{{W_i}}}{{{u_n}}} > 1} \right) = \overline F \left( {{u_n}} \right),                                                                                                          \\
    \E\left[{\phi _2}\left( {{W_{n,i}}} \right)\right] & = \P\left( {\frac{{{W_i}}}{{{u_n}}} > e} \right) = \overline F \left( {e{u_n}} \right),                                                                                                         \\
    {v_n}                                              & = \P\left( {{W_{n,i}} \ne 0} \right) = \P\left( {\frac{{{W_i}}}{{{u_n}}} > 1} \right) = \overline F \left( {{u_n}} \right).
\end{align*}

Let \(w_n=\overline F \left( {e{u_n}} \right)\), \({\lambda _n} = {{\overline F
        \left( {{u_n}} \right)}}/{{\overline F \left( {e{u_n}} \right)}}={v_n}/{w_n}\)
(notice that \(\lambda_n\to e^\beta\)) and \({y_n} = \sqrt
{{{{v_n}}}/{{(n{w_n}^2)}}}\), then,
\begin{align}
    \widehat{\lambda}_n & =\frac{{\sum\limits_{i = 1}^n {\mathbb{I}\left\{ {{W_i} > u_n} \right\}} }}{{\sum\limits_{i = 1}^n {\mathbb{I}\left\{ {{W_i} > eu_n} \right\}} }} = \frac{{n\E\left[{\phi_1}\left( {{W_{n,1}}} \right)\right] + \sum\limits_{i = 1}^n {\left\{ {{\phi _1}\left( {{W_{n,i}}} \right) - \E\left[{\phi _1}\left( {{W_{n,1}}} \right)\right]} \right\}} }}{{n\E\left[{\phi _2}\left( {{W_{n,1}}} \right)\right] + \sum\limits_{i = 1}^n {\left\{ {{\phi _2}\left( {{W_{n,i}}} \right) - \E\left[{\phi _2}\left( {{W_{n,1}}} \right)\right]} \right\}} }}\nonumber                                                                                                                                                                                                                                                                                                  \\
                        & = \frac{{\frac{{\E\left[{\phi _1}\left( {{W_{n,1}}} \right)\right]}}{{\E\left[{\phi _2}\left( {{W_{n,1}}} \right)\right]}} + \frac{{\sum\limits_{i = 1}^n {\left\{ {{\phi _1}\left( {{W_{n,i}}} \right) - \E\left[{\phi _1}\left( {{W_{n,1}}} \right)\right]} \right\}} }}{{n\E\left[{\phi _2}\left( {{W_{n,1}}} \right)\right]}}}}{{1 + \frac{{\sum\limits_{i = 1}^n {\left\{ {{\phi _2}\left( {{W_{n,i}}} \right) - \E\left[{\phi _2}\left( {{W_{n,1}}} \right)\right]} \right\}} }}{{n\E\left[{\phi _2}\left( {{W_{n,1}}} \right)\right]}}}} = \frac{{\frac{{\E\left[{\phi _1}\left( {{W_{n,1}}} \right)\right]}}{{\E\left[{\phi _2}\left( {{W_{n,1}}} \right)\right]}} + {{\widetilde Z}_n}\left( {{\phi _1}} \right)\sqrt {\frac{{{v_n}}}{{n{w_n}^2}}} }}{{1 + {{\widetilde Z}_n}\left( {{\phi _2}} \right)\sqrt {\frac{{{v_n}}}{{n{w_n}^2}}} }}\nonumber \\
                        & = \frac{{{\lambda _n} + {{\widetilde Z}_n}\left( {{\phi _1}} \right){y_n}}}{{1 + {{\widetilde Z}_n}\left( {{\phi _2}} \right){y_n}}}. \label{lamda_n}
\end{align}

In order to apply Lemma \ref{dress_and_rotzen}, we need to check its hypotheses. We
will start by \AssumpAOne. By hypothesis,
\(e^{k_n\beta}{\ln{n}}/{n}=o\left(L\left(e^{k_n+1}\right)\right)\) and \(L\) is
slowly varying, hence, we can write \({e^{{k_n}\beta }}{{\ln n}}/{n} = L\left(
{{e^{{k_n}}}} \right)\varepsilon \left( n \right)\) where \(\varepsilon \left( n
\right)\to 0\), then
\begin{equation*}
    nv_n=n\overline F ( {{u_n}} )=n\frac{L(e^{k_n})}{e^{k_n\beta}}=\frac{\ln n}{\epsilon(n)}\to +\infty.
\end{equation*}

Hypotheses {\AssumpATwo} follows directly from the following calculations
\begin{align*}
    \E \left[ \phi_1(W_{n,1})^4  \right] & = \E \left[ \phi_1(W_{n,1})\right]=\P\left( W_i>u_n \right)=\overline{F}(u_n)=v_n,         \\
    \E \left[ \phi_2(W_{n,1})^4  \right] & = \E \left[ \phi_2(W_{n,1})\right]=\P\left( W_i>e u_n \right)=\overline{F}(e u_n)\leq v_n. \\
\end{align*}

Finally, for {\AssumpAThree}, observe that
\begin{align*}
    \E\left[ {{\phi _1}\left( {{W_{n,1}}} \right)^2} \right]                                  & = v_n, \\
    \E\left[ {{\phi _1}\left( {{W_{n,1}}} \right){\phi _2}\left( {{W_{n,1}}} \right)} \right] & = w_n, \\
    \E\left[ {{\phi _2}\left( {{W_{n,1}}}\right)^2} \right]                                   & = w_n,
\end{align*}
therefore, the limits stated in {\AssumpAThree} exist and their values are \(\sigma_{11}=1\), \(\sigma _{12}=\sigma _{22}=e^{-\beta}\).

Applying Lemma \ref{dress_and_rotzen} we obtain that \({( {\widetilde {{Z}}_n\left(
                {{\phi _k}} \right)} )_{1 \leqslant k \leqslant 2}}\) converges to a centred normal
distribution with covariance matrix \({\left( {{\sigma _{kl}}} \right)_{1 \leqslant
            k,l \leqslant 2}}\). Taking into account that \({y_n}\sim{{{e^\beta }}}/{{\sqrt
        {n{v_n}} }}\), it follows that
\begin{align}
    {\widehat \lambda _n} & = \left( {{\lambda _n} + {{\widetilde Z}_n}\left( {{\phi _1}} \right){y_n}} \right)\left( {1 - {{\widetilde Z}_n}\left( {{\phi _2}} \right){y_n} + o_P\left( {\frac{1}{{\sqrt {n{v_n}} }}} \right)} \right)\nonumber        \\
                          & = {\lambda _n} + {y_n}\left( {{{\widetilde Z}_n}\left( {{\phi _1}} \right) - {\lambda _n}{{\widetilde Z}_n}\left( {{\phi _2}} \right)} \right) + {o_P}\left( {\frac{1}{{\sqrt {n{v_n}} }}} \right).\label{eq:linearization}
\end{align}

Then, \(\sqrt {n{v_n}} ( {{{\widehat \lambda }_n} - {\lambda _n}} )\) converges
weakly to a centred normal distribution with variance \({e^{2\beta} }\left( {{\sigma
        _{11}} + {e^{2\beta }}{\sigma _{22}} - 2{e^\beta }{\sigma _{12}}} \right) =
    {e^{2\beta }}\left( {{e^\beta } - 1} \right)\). This can be resumed in the following
lemma.

\begin{lemma}\label{weakly_convergence_ratio} Let \(W_n\) and \(u_n\) be as in Lemma \ref{dress_and_rotzen},
    if \(k_n\) satisfies the conditions of Theorem \ref{consistency}, then the following convergence in distribution holds
    \begin{equation*}
        \sqrt {n\bar F\left( {{e^{{k_n}}}} \right)} \left( {\frac{{\sum\limits_{i = 1}^n {\mathbb{I}\left\{ {{W_i} > {e^{{k_n}}}} \right\}} }}{{\sum\limits_{i = 1}^n {\mathbb{I}\left\{ {{W_i} > {e^{{k_n} + 1}}} \right\}} }} - \frac{{\bar F\left( {{e^{{k_n}}}} \right)}}{{\bar F\left( {{e^{{k_n} + 1}}} \right)}}} \right) \Rightarrow \mathcal{N}\left(0,\; e^{2\beta}(e^{\beta}-1)\right).
    \end{equation*}
\end{lemma}

Lemmas \ref{limit_alpha}, \ref{log_transform_weakly_conv} and
\ref{weakly_convergence_ratio} combined with equation (\ref{beta_k_eq}) imply the
first part of Theorem \ref{asymptotic_normality}, the second part follows from Lemma
\ref{convergence_probability_p_k} and Slutsky's Theorem. Corollary
\ref{asymptotic_normality_real_value} follows immediately.

\subsection{Technical details of Example \ref{ex:matsui}}\label{sec:details_example_matsui}

To establish equation \eqref{eq:svf_bound_matsui}, we proceed as follows. For any
\(x>1\), we can express:
\begin{equation*}
    \frac{L(x)}{L(ex)}=\left( \frac{10^l x}{ \lfloor 10^lx \rfloor + 1 } \frac{ \lfloor 10^l ex \rfloor + 1 }{10^l ex}  \right)^\beta = e^{-\beta}\left( \frac{ \lfloor 10^l ex \rfloor + 1 }{ \lfloor 10^lx \rfloor + 1 } \right)^\beta.
\end{equation*}

Using that \(y-1 < \lfloor y \rfloor \leq y\) for any \(y>0\), we can derive:
\begin{equation*}
    \frac{ \lfloor 10^l ex \rfloor + 1 }{ \lfloor 10^lx \rfloor + 1 }\leq \frac{10^l ex + 1}{10^lx}\leq e + \frac{1}{10^l x},
\end{equation*}
consequently,
\begin{equation*}
    \frac{L(x)}{L(ex)}\leq \left(1+\frac{1}{10^lex}\right)^\beta = 1+\frac{\beta}{10^le x} + o(x^{-1}).
\end{equation*}
By a similar argument, if follows that:
\begin{equation*}
    \frac{L(x)}{L(ex)}\geq \left(1-\frac{1}{10^lx+1}\right)^\beta = 1-\frac{\beta}{10^lx+1} + o(x^{-1}).
\end{equation*}
As a result, for sufficiently large \(x\), we obtain that
\begin{equation*}
    \left| \frac{L(x)}{L(ex)}-1 \right|\leq \frac{2\beta}{10^l}x^{-1}.
\end{equation*}

\subsection{Proof of Lemma \ref{bias_sr2}}

The representation is a direct application of Lemma \ref{ln_ratio} and the fact that
\(g\) is a regularly varying function of index \(\rho\).

\begin{lemma}\label{ln_ratio} Assume that \(L\) satisfies \(SR2\), has
    positive decrease and \(x\) is big enough such that representation the (\ref{sr2_repre}) holds, then
    \begin{equation}
        \ln\left(\frac{L(x)}{L(\lambda x)}\right) = -c|\rho|^{-1}\left( g(x)-g(\lambda x) \right) + o\left(g(x)\right).
    \end{equation}
\end{lemma}
\begin{proof}
    Denote \(A(x) = c\rho^{-1}g(x)+o\left(g \left(x\right) \right)\). By
    \eqref{sr2_repre} we have
    \begin{align}
        \ln\left(\frac{L(x)}{L(\lambda x)}\right) & = \ln\left( \frac{C(1+A(x))}{C(1+A(\lambda x))} \right) = \ln\left( \frac{1+A(x)}{1+A(\lambda x)} \right)\nonumber \\
                                                  & = \ln{\left(1+A\left(x\right)\right)}-\ln(1+A(\lambda x)).\label{bias_expansion_in_terms_a}
    \end{align}

    Using the first order expansion for \(\ln(1+A(x))\) we have that
    \begin{align}
        \ln\left(1+A\left(x\right)\right) & =c\rho^{-1}g(x)+o\left(g\left(x\right)\right)+\underbrace{o\left(c\rho^{-1}g(x)+o\left(g\left(x\right)\right)\right)}_{o\left(g\left(x\right)\right)}\nonumber \\
                                          & =c\rho^{-1}g(x)+o\left(g\left(x\right)\right).\label{ln_1_plus_a_expansion}
    \end{align}

    Applying (\ref{ln_1_plus_a_expansion}) to \(\lambda x\) we get
    \begin{equation}
        \ln\left(1+A\left(\lambda x\right)\right) = c\rho^{-1}g(\lambda x)+o\left(g\left(x\right)\right),\label{ln_1_plus_a_expansion_lambda}
    \end{equation}
    where we have used that if \(g\) is regularly varying then
    \(o\left(g\left(\lambda x\right)\right)=o\left(g\left(x\right)\right)\). The
    result now follows by plugging (\ref{ln_1_plus_a_expansion}) and
    (\ref{ln_1_plus_a_expansion_lambda}) into (\ref{bias_expansion_in_terms_a}).
\end{proof}

\subsection{Proof of Theorem \ref{th:rate_convergence}}

Theorem \ref{th:rate_convergence} is the the result of the Continuous Mapping Theorem
and the following Theorem, which is the particularization of Theorem 2.3 in
\cite{Chen1999} to the indicator function \(\I_B\). In section \ref{sec:funclimit},
we provide a functional generalization of Theorem \ref{thm:weaklimit}.

\begin{theorem}[Limit distributions] \label{thm:weaklimit} Let $\beta\in[0,1)$
    and $\nu$ be any probability distribution on $E$. Suppose that the chain $X$ is
    $\beta$-regular and let $B$ be a Harris set with finite and strictly
    positive $\mu$-measure. We have the following convergences
    in $\mathbb{P}_{\nu}$-distribution.
    \begin{itemize}
        \item[(i)] If $\beta=0$, we then have, as $n\rightarrow \infty$,
              \[\frac{1}{L_{\nu, D}(n)}\Sigma_n(B)\Rightarrow \mathcal{E}(1/\mu(B))
                  \text{ in } \mathbb{P}_{\nu}\text{-distribution},\]
              where $\mathcal{E}(\lambda)$ denotes the exponential distribution with
              mean $1/\lambda>0$.
        \item[(ii)] If $\beta \in (0,1)$, we have, as $n\rightarrow \infty$,
              \[\frac{1}{n^{\beta}L_{\nu, D}(n)}\Sigma_n(B)\Rightarrow \mu(B)/(Z_{\beta})^{\beta}
                  \text{ in } \mathbb{P}_{\nu}\text{-distribution},\]
              where $Z_{\beta}$ is a stable random variable with Laplace transform
              \[\psi_{\beta}\left ( t \right )=\exp\left (-t^{\beta}/\Gamma( \beta+1) \right ),\quad t\geq 0.\]
    \end{itemize}
\end{theorem}

\subsection{Proof of Corollary \ref{cor:markov:deterministic_control}}
The random variable \(\frac{M_{\beta}(1)}{\Gamma(1-\beta)} \) is continuous and positive, therefore,
we can find positive constants \(R_\delta\) and \(H_\delta\) such that
\begin{equation}\label{eq:markov:prob_mittag_leffler}
    \P\left( {\frac{M_{\beta}(1)}{\Gamma(1-\beta)} \in \left[ {R_\eta ,H_\eta} \right]} \right) \geqslant 1 - \frac{\delta }{2}.
\end{equation}
By Proposition \ref{prop:markov:limit_dist}, \(N_n\overline{F}(n)\) converges in distribution
to \(\frac{M_{\beta}(1)}{\Gamma(1-\beta)}\), hence, there exists \(n_\delta\) such that
\begin{equation}\label{eq:markov:convergence_mittag_leffler}
    \left|\P \left( N_n\overline{F}(n) \in \left[ {R_\eta ,H_\eta} \right]\right) - \P \left( {\frac{M_{\beta}(1)}{\Gamma(1-\beta)} \in \left[ {R_\eta ,H_\eta} \right]} \right)\right|<\frac{\delta}{2}\quad\forall n\geq n_\delta.
\end{equation}
The result now follows by combining \eqref{eq:markov:prob_mittag_leffler} and \eqref{eq:markov:convergence_mittag_leffler}.
\subsection{Proof of Proposition \ref{prop:markov:bias_variance_thm}}

As in Proposition \ref{bias_variance_thm}, the result is a direct consequence of
the Lemma \ref{lemma:markov:finite_bound} and the inequality
\begin{equation*}
    \left|\widehat{\beta}_{N_n}(k)-\beta\right| \leq
    \left|\widehat{\beta}_{N_n}(k)-\beta(k)\right| + \left| \ln\left(\frac{L(e^{k})}{L(e^{k+1})}\right) \right|.
\end{equation*}

\begin{lemma}\label{lemma:markov:bernstein}If \(0<t<p_k\) then,
    \begin{equation}
        \P \left(\left| \widehat{p}^{(n)}_k-p_k \right|>t \right)\leq 2 \exp{\left(-\frac{nt^2}{4p_k}\right)}
    \end{equation}
\end{lemma}
\begin{proof}
    By Bernstein's inequality \cite[Lemma 19.32]{Vaart2000}, for all \(y>0\)
    \begin{align*}
        \P \left(\sqrt{n} \left| \widehat{p}^{(n)}_k-p_k \right|>y \right) & \leq 2 \exp{\left(-\frac{1}{4}\frac{y^2}{p_k+t/\sqrt{n}}\right)}                                                      \\
                                                                           & \leq  2 \min \left( \exp\left(-\frac{1}{4}\frac{y^2}{p_k}\right),\exp\left(-\frac{1}{4}\sqrt{n}{p_k}\right)  \right),
    \end{align*}
    and if \(y<\sqrt{n}p_k\), then
    \begin{equation*}
        \P \left(\sqrt{n} \left| \widehat{p}^{(n)}_k-p_k \right|>y \right)\leq 2  \exp\left(-\frac{1}{4}\frac{y^2}{p_k}\right).
    \end{equation*}
    The result now follows by setting \(y=t\sqrt{n}\).
\end{proof}

\begin{lemma}\label{lemma:markov:finite_bound}
    Let \(\delta>0\) and \(n\geq n_\delta\) such that \(p_{k+1}\geq {({40H_\delta}R^{-1}_\delta)}^2 v_n(\delta)\),
    then
    \begin{equation}\label{eq:markov:varianceBoundEq}
        \P_\nu \left( \left|\widehat{\beta}_{N_n}(k)-\beta(k)\right| \leq \frac{60 H_\delta}{R_{\delta}}\sqrt{\frac{v_n(\delta)}{p_{k+1}}} \right)>1-2\delta,
    \end{equation}
    where \(H_\delta\), \(R_\delta\) and \(n_\delta\) are as in Corollary \ref{cor:markov:deterministic_control}.
\end{lemma}
\begin{proof}
    Let \(\delta>0\) be fixed. With \(R_\delta\) and \(H_\delta\) as in Corollary \ref{cor:markov:deterministic_control}, define the events
    \begin{equation*}
        \mathcal{E}_n = \left\{ N_n \in \left[\frac{R_\delta}{\overline{F}(n)}, \frac{H_\delta}{\overline{F}(n)}\right] \right\},
    \end{equation*}
    and denote by \(\overline{\mathcal{E}}_n\) the complement of \(\mathcal{E}_n\).
    By Corollary \ref{cor:markov:deterministic_control}, there exists \(n_\delta\) such that
    \(\P(\overline{\mathcal{E}}_n)<\delta/2\). Let \(t>0\), then, for each \(n\geq n_\delta\) it
    holds that
    \begin{align*}
        \P\left( \left| \widehat{p}^{(N_n)}_k - p_k \right|>t \right) & \leqslant \P\left(\left\{\left| \widehat{p}^{(N_n)}_k - p_k \right|>t\right\}\cap \mathcal{E}_n\right) + \P\left(\overline{\mathcal{E}}_n\right), \\
                                                                      & \leqslant \P\left(\left\{\left| \widehat{p}^{(N_n)}_k - p_k \right|>t\right\}\cap \mathcal{E}_n\right) + \frac{\delta}{2}.
    \end{align*}
    Notice that, on \(\mathcal{E}_n\) we have the following inclusions
    \begin{align*}
        \left\{ \left| \widehat{p}^{(N_n)}_k - p_k \right|>t \right\} & = \left\{ \frac{1}{N_n} \left| \sum_{i=1}^{N_n}{\I\left\{ S_i>e^k \right\}} - p_k \right|>t \right\}                                     \nonumber                                                                \\
                                                                      & = \left\{ \left| \sum_{i=1}^{N_n}{\I\left\{ S_i>e^k \right\}} - p_k \right|>N_n t \right\}                                               \nonumber                                                                \\
                                                                      & \subseteq \left\{ \max_{1\leq m\leq \frac{H_\delta}{\overline{F}(n)}} \left| \sum_{i=1}^{m}{\I\left\{ S_i>e^k \right\}} - p_k \right| > \frac{R_\delta t}{\overline{F}(n)} \right\},\label{eq:markov:inclussions}
    \end{align*}
    therefore,
    \begin{equation*}
        \P\left(\left\{\left| \widehat{p}^{(N_n)}_k - p_k \right|>t\right\}\cap \mathcal{E}_n\right) \leqslant \P \left( \max_{1\leq m\leq \frac{H_\delta}{\overline{F}(n)}} \left| \sum_{i=1}^{m}{\I\left\{ S_i>e^k \right\}} - p_k \right| > \frac{R_\delta t}{\overline{F}(n)} \right).
    \end{equation*}
    By Montgomery-Smith's inequality \cite[Corollary 4]{montgomery1993comparison}, the right-hand side of the previous
    inequality is smaller than
    \begin{align*}
        9\P\left( \left| \sum_{i=1}^{\frac{H_{\delta}}{\overline{F}(n)}}{\I\left\{ S_i>e^k \right\}} - p_k \right| > \frac{R_\delta t}{10\overline{F}(n)} \right) = 9\P\left( \left| \widehat{p}^{ \left(\frac{H_{\delta}}{\overline{F}(n)}\right) }_k-p_k \right| > \frac{R_{\delta}t}{10H_{\delta}} \right)
    \end{align*}
    Using Lemma \ref{lemma:markov:bernstein} we obtain that, for \(t<\frac{10p_k H_\delta }{R_\delta}\) it holds that
    \begin{align*}
        \P\left( \left| \widehat{p}^{\left(\frac{H_{\delta}}{\overline{F}(n)}\right)}_k-p_k \right| > \frac{R_{\delta}t}{10H_{\delta}} \right) \leq 2 \exp\left(-\frac{R^2_{\delta}}{400H_{\delta}}\frac{t^2}{\overline{F}(n)p_k}\right).
    \end{align*}
    Therefore for all \(t<\frac{10p_k H_\delta }{R_\delta}\) we have that
    \begin{equation}\label{eq:markov:bound_t}
        \P\left( \left| \widehat{p}^{(N_n)}_k - p_k \right|>t \right)\leq 18\exp\left(-\frac{R^2_{\delta}}{400H_{\delta}}\frac{t^2}{\overline{F}(n)p_k}\right)+\frac{\delta}{2}
    \end{equation}
    Taking \(t=20R^{-1}_{\delta}\sqrt{\ln{(36/\delta)} \overline{F}(n)p_k H_\delta }=20 H_\delta R^{-1}_\delta\sqrt{v_n(\delta)p_k}\) in \eqref{eq:markov:bound_t}
    implies that, if \(n\) is large enough such that \(p_k>4v_n(\delta)\)
    and \(n\geq n_\delta\), then
    \begin{equation}\label{eq:markov:bound_t}
        \P\left( \left| \widehat{p}^{(N_n)}_k - p_k \right|> 20\frac{H_\delta}{R_\delta}\sqrt{v_n p_k} \right)\leq \delta.
    \end{equation}
    Equation \eqref{eq:markov:varianceBoundEq} now follows by the same argument used in the proof
    of Proposition \ref{bias_variance_thm}.
\end{proof}

\subsection{Proof of Theorem \ref{th:consistency_markov}}

The recurrence of the chain implies that \(N_n\) converges almost surely to
\(+\infty\), then, by Theorem 8.1 in page 302 of \cite{Gut2013}, we can replace \(n\)
by \(N_n\) on the strong consistency results we presented on section
\ref{sec:tail-index-estimation}, to obtain equivalent results for the sequence
\(S_1\ldots,S_{N_n}\).

\section{Averaged Estimators}\label{subsec:ave}

Here we collect some remarks and results related to the averaged estimator $\widehat
    \beta_n \left( {k,m} \right)$. First, we detail how to get the expression
\eqref{bias_average_eq} from \eqref{beta_k_eq}. Let \(k>0\) be fixed, for each $j$ we
have:
\begin{equation*}
    \beta \left( {k + j} \right) = \beta  + \ln \left( {\frac{{L\left( {{e^{k + j}}} \right)}}{{L\left( {{e^{k + j + 1}}} \right)}}} \right),
\end{equation*}
then,
\begin{align*}
    \frac{1}{{2m + 1}}\sum\limits_{j =  - m}^m {\beta \left( {k + j} \right)} & = \frac{1}{{2m + 1}}\sum\limits_{j =  - m}^m \beta   + \frac{1}{{2m + 1}}\sum\limits_{j =  - m}^m {\ln \left( {\frac{{L\left( {{e^{k + j}}} \right)}}{{L\left( {{e^{k + j + 1}}} \right)}}} \right)} \\                                             & = \beta+\frac{1}{{2m + 1}}\ln \left( {\prod\limits_{j =  - m}^m {\frac{{L\left( {{e^{k + j}}} \right)}}{{L\left( {{e^{k + j + 1}}} \right)}}} } \right)                                              \\
                                                                              & = \beta + \frac{1}{{2m + 1}}\ln \left( {\frac{{L\left( {{e^{k - m}}} \right)}}{{L\left( {{e^{k + m + 1}}} \right)}}} \right).
\end{align*}

Our first result in this regard is the concentration inequality equivalent to
\eqref{rateConvergence} but for the averaged version of the estimator.

\begin{proposition}\label{bias_variance_thm_average}
    Let \(k\) and \(m\) such that \(k>m\) and let \(\delta\in (0,\; 1/(2(1+2m))\).
    Then, as soon as \(p_{k+m+1}\geq 16 u_n(\delta)\), we have with probability
    larger than \(1-2\delta(1+2m)\):
    \begin{equation}\label{rateConvergenceAverage}
        \left| {\widehat \beta_n \left( {k,m} \right) - \beta } \right| \leqslant 6\sqrt {\frac{{{u_n}\left( \delta  \right)}}{{{p_{k + m + 1}}}}}  + \frac{1}{{2m + 1}}\left| {\ln \left( {\frac{{L\left( {{e^{k - m}}} \right)}}{{L\left( {{e^{k + m + 1}}} \right)}}} \right)} \right|.
    \end{equation}
\end{proposition}

The following results show that, for well-chosen \(k_n\) and \(m_n\), the estimator
\(\widehat{\beta}_n(k_n,m_n)\) is strongly consistent.

\begin{theorem}[Strong consistency]\label{consistency_average}
    Let \(k_n\) and \(m_n\) such that, as $n\to\infty$
    \begin{itemize}
        \item[i)] \({k_n} - {m_n} \to + \infty\),
        \item[ii)] \(\sum\limits_{n = 1}^{ + \infty } {\frac{4}{{{n^2}}}\left( {1 -
                  2{m_n}} \right)}\) is convergent,
        \item[iii)] \({e^{\left( {{k_n} + {m_n}} \right)\beta }}\frac{{\ln n}}{n} =
              o\left( {L\left( {{e^{{k_n} + {m_n}}}} \right)} \right);\)
    \end{itemize}
    then, \(\widehat \beta_n \left( {k_n,m_n} \right)\) converges almost surely to \(\beta\).
\end{theorem}

\begin{corollary}\label{consistency_ln_average}
    Let \(A,l\) be a positive numbers such that \(l>1\) and $l>{A\beta}/{(1-A\beta)}$ then
    \begin{equation*}
        \widehat \beta_n \left( {A\ln n,\frac{{A\ln n}}{l}} \right) \to \beta\quad a.s.
    \end{equation*}
\end{corollary}

\subsection{Proof of Proposition \ref{bias_variance_thm_average}}

Similarly to the proof of Theorem \ref{bias_variance_thm}, Theorem
\ref{bias_variance_thm_average} follows by triangular inequality, the definition of
$\beta \left( {k,m} \right) $ and the following Lemma
\ref{varianceBoundLemma_average}, which provides us a bound for the difference
between \(\widehat{\beta}(k,m)-\beta(k,m)\).

\begin{lemma}\label{varianceBoundLemma_average}
    Let \(\delta>0\) and \(k\) and \(m\) such that \(p_{k+m+1}\geq
    16u_n(\delta)\), then
    \begin{equation}\label{varianceBoundEq_average}
        \left| {\widehat \beta_n (k,m) - \beta (k,m)} \right| \leqslant 6\sqrt {\frac{{{u_n}(\delta )}}{{{p_{k + m + 1}}}}} ,
    \end{equation}
    with probability larger than \(1 - 2\delta \left( {1 - 2m} \right)\).
\end{lemma}
\begin{proof}
    For the left-hand side of equation (\ref{varianceBoundEq_average}) we have
    \begin{align*}
        \left| {\widehat \beta_n (k,m) - \beta (k,m)} \right| & = \frac{1}{{2m + 1}}\left| {\sum\limits_{j =  - m}^m {\left( {\widehat \beta_n \left( {k + j} \right) - \beta \left( {k + j} \right)} \right)} } \right|          \\
                                                              & = \frac{1}{{2m + 1}}\left| {\sum\limits_{j = 0}^{2m} {\left( {\widehat \beta_n \left( {k - m + j} \right) - \beta \left( {k - m + j} \right)} \right)} } \right|,
    \end{align*}
    then,
    \begin{equation}\label{eqn_lhh_diff_b_hat_b}
        \left| {\widehat \beta_n (k,m) - \beta (k,m)} \right| \leqslant \frac{1}{{2m + 1}}\sum\limits_{j = 0}^{2m} {\left| {\widehat \beta_n \left( {k - m + j} \right) - \beta \left( {k - m + j} \right)} \right|}.
    \end{equation}
    Given \(p_{k+m+1}\geq 16 u_n(\delta)\), it follows that \({p_{k - m + j +
                    1}} \geqslant 16{u_n}(\delta )\) for all \(j\) between 0 and \(2m\),
    allowing the application of Lemma \ref{varianceBoundLemma}. This yields, for
    each \(j\),
    \begin{equation}\label{particularBound_j}
        \left| {\widehat \beta_n \left( {k - m + j} \right) - \beta \left( {k - m + j} \right)} \right| \leqslant 6\sqrt {\frac{{{u_n}(\delta )}}{{{p_{k - m + j + 1}}}}},
    \end{equation}
    with probability bigger than \(1-2\delta\).

    The joint probability of inequality \eqref{particularBound_j} holding for all
    \(j\) between \(0\) and \(2m\) is greater than \(\left( {2m + 1} \right)\left( {1
        - 2\delta } \right) - 2m = 1 - 2\delta \left( {1 - 2m} \right)\). Hence, with at
    least this probability, equation \eqref{eqn_lhh_diff_b_hat_b} simplifies to,
    \[\left| {\widehat \beta_n (k,m) - \beta (k,m)} \right| \leqslant \frac{1}{{2m + 1}}\sum\limits_{j = 0}^{2m} {6\sqrt {\frac{{{u_n}(\delta )}}{{{p_{k - m + j + 1}}}}} }  \leqslant 6\sqrt {\frac{{{u_n}(\delta )}}{{{p_{k + m + 1}}}}}.\]

\end{proof}

\subsection{Proof of Theorem \ref{consistency_average}}

The following Lemma \ref{limit_alpha_average} shows that if \(k_n-m_n\to +\infty\),
then $\beta \left( {{k_n},{m_n}} \right) \to \beta$. Theorem
\ref{consistency_average} now follows by the same argument used to prove Theorem
\ref{consistency}, using the convergence of $\beta \left( {{k_n},{m_n}} \right)$
instead of Lemma \ref{limit_alpha} and Lemma \ref{varianceBoundLemma_average} instead
of Lemma \ref{varianceBoundLemma}.

\begin{lemma}\label{limit_alpha_average}
    Let \(\alpha_n\), \(k_n\) and \(b_n\) be sequences such that,
    \(\alpha_n\to\alpha\), \(k_n\to +\infty\) and \(k_n-n_n\to +\infty\).
    Then,
    \[\frac{1}{{2{b_n} + 1}}\sum\limits_{j =  - {b_n}}^{{b_n}} {{\alpha _{{k_n} + j}}}  \to \alpha.\]
\end{lemma}
\begin{proof}
    Let \({\rho_n} = \frac{1}{{2{b_n} + 1}}\sum\limits_{j = - {b_n}}^{{b_n}}
    {{\alpha _{{k_n} + j}}}\), then
    \begin{equation*}
        {\rho_n} - \alpha  = \frac{1}{{2{b_n} + 1}}\sum\limits_{j =  - {b_n}}^{{b_n}} {\left( {{\alpha _{{k_n} + j}} - \alpha } \right)}  = \frac{1}{{2{b_n} + 1}}\sum\limits_{j = 0}^{2{b_n}} {\left( {{\alpha _{{k_n} - {b_n} + j}} - \alpha } \right)}.
    \end{equation*}

    Fix \(\epsilon > 0\). The convergence of \(\alpha_n\) ensures the existence of
    \(N_1\) such that \(\left| {{\alpha _n} - \alpha } \right| < \epsilon\) for all
    \(n \geq {N_1}\). Given that \(k_n-b_n\to +\infty\), there exists \(N_2\)
    satisfying \({k_n} - {b_n} \geq {N_1},\) for all \(n\geq N_2\), which yields
    \(\left| {{\alpha _{{k_n} - {b_n} + j}} - \alpha } \right| \leq \epsilon\) for
    all \(n\geq N_2\) and \(j\in\mathbb{N}\). Consequently, for \(n\geq N_2\),
    \begin{equation*}
        \left| {{\rho_n} - \alpha } \right| \leq \frac{1}{{2{b_n} + 1}}\sum\limits_{j = 0}^{2{b_n}} {\left| {{\alpha _{{k_n} - {b_n} + j}} - \alpha } \right|}  \leq \frac{1}{{2{b_n} + 1}}\sum\limits_{j = 0}^{2{b_n}}\epsilon=\epsilon.
    \end{equation*}
\end{proof}

\subsection{Proof of Corollary \ref{consistency_ln_average}}

We just need to show that sequences \(k_n=A\ln n\) and \(m_n={A\ln n}/{l}\) satisfy
conditions (i), (ii) and (iii) of Theorem \ref{consistency_average}. The first two
are trivially satisfied, for the third one, notice that
\begin{equation*}
    \mathop {\lim }\limits_n \frac{{{e^{\left( {A\ln n + \frac{{A\ln n}}{l}} \right)\beta }}}}{{L\left( {{e^{A\ln n + \frac{{A\ln n}}{l}}}} \right)}}\frac{{\ln n}}{n} = \mathop {\lim }\limits_n \frac{1}{{{n^{\frac{{1 - \left( {1 + \frac{1}{l}} \right)A\beta }}{2}}}L\left( {{n^{\left( {1 + \frac{1}{l}} \right)A}}} \right)}}\frac{{\ln n}}{{{n^{\frac{{1 - \left( {1 + \frac{1}{l}} \right)A\beta }}{2}}}}}.
\end{equation*}

The condition $l>{A\beta}/{(1-A\beta)}$ implies that \(1 - \left( {1 + {1}/{l}}
\right)A\beta > 0\), therefore,
\begin{equation*}
    \mathop {\lim }\limits_n \frac{1}{{{n^{\frac{{1 - \left( {1 + \frac{1}{l}} \right)A\beta }}{2}}}L\left( {{n^{\left( {1 + \frac{1}{l}} \right)}}} \right)}} = 0\quad \textnormal{and}\quad \mathop {\lim }\limits_n \frac{{\ln n}}{{{n^{\frac{{1 - \left( {1 + \frac{1}{l}} \right)A\beta }}{2}}}}} = 0,
\end{equation*}
which shows that \(k_n\) and \(m_n\) satisfy condition (iii) in Theorem
\ref{consistency_average}.

\section{Tail index estimation in the positive recurrent case}\label{subsec:positive_recurrent_case}

In this section, we prove the asymptotic normality of our estimator in the positive
recurrent case. The main result is the following

\begin{theorem}[Asymptotic normality in the positive recurrent case]\label{th:asymptotic_normality_positive_recurrent}
    Suppose $X$ is a regenerative positive recurrent Markov chain such that the distribution
    of its regeneration time satisfies equation \eqref{eq:Zipf}. Assume that \(k_n\) satisfies
    the hypothesis of Theorem \ref{asymptotic_normality} as well as the following extra assumption:
    \begin{itemize}
        \item[\AssumpVelocityText]\label{assumpt:AssumpVelocity} There exists $D^\prime>0$ such that for any $0<D<D^{\prime}$
              \begin{equation}
                  \sum_{j=n+1}^{n\left(1+D\right)}{\bar{F}\left(e^{k_j}\right)}= nD\bar{F}\left(e^{k_n}\right)+o\left(\bar{F}\left(e^{k_n}\right)\right).
              \end{equation}
    \end{itemize}
    Then,
    \begin{itemize}
        \item[(i)] as $n\to +\infty$, we have the convergence in distribution:
              \begin{equation*}
                  \sqrt {{N_n} p_{k_{N_n}}} \left(\widehat \beta_{N_n} ({k_{N_n}}) - \beta(k_{N_n}) \right) \Rightarrow \mathcal{N}\left(0,\; e^{\beta}-1\right).
              \end{equation*}
        \item[(ii)]In addition, asymptotic normality holds true for the 'standardized' deviation:
              \begin{equation*}
                  \frac{{\sqrt {{N_n}\widehat p_{{k_{N_n}}}^{({N_n})}} \left( {{{\widehat \beta }_{N_n}}\left( {{k_{N_n}}} \right) - \beta \left( {{k_{N_n}}} \right)} \right)}}{{\sqrt {{e^{{{\widehat \beta }_{N_n}}\left( {{k_{N_n}}} \right)}} - 1} }}\Rightarrow \mathcal{N}\left( {0,\; 1} \right), \text{ as } n\to +\infty.
              \end{equation*}
    \end{itemize}
\end{theorem}

\begin{remark} Assumption {\AssumpVelocity} can be replaced by the following slightly more
    restrictive, but easier to verify, assumption:
    \begin{itemize}
        \item[\AssumpVelocityPrimeText]\label{assumpt:AssumpVelocityPrime} There exists $D^\prime>0$ such that for any $0<D<D^{\prime}$
              \begin{equation*}
                  \bar{F}\left(e^{k_{n(1+D)}}\right)=\bar{F}\left(e^{k_n}\right)+o\left(\frac{\bar{F}\left(e^{k_n}\right)}{n}\right).
              \end{equation*}
    \end{itemize}
\end{remark}

As in the case of Theorem \ref{asymptotic_normality}, the main result follows
directly from the following lemma, which is an extension of Lemma
\ref{weakly_convergence_ratio} for the case where the number of i.i.d samples we have
is random.

\begin{lemma}\label{lemma:dress_and_rotzen} Let \(W_n\) be a sequence of i.i.d. random variables with survival
    function (\ref{eq:Zipf}), and $k_n$ be a sequence that satisfies the hypothesis of Theorem \ref{th:asymptotic_normality_positive_recurrent}.
    Suppose that $T_n$ is a sequence of positive, integer-valued random variables such that
    $\frac{T_n}{n}$ converges in probability to some positive number $\theta$. Then
    \begin{equation*}
        \sqrt {T_n\bar F\left( {{e^{{k_{T_n}}}}} \right)} \left( {\frac{{\sum\limits_{i = 1}^{T_n} {\mathbb{I}\left\{ {{W_i} > {e^{{k_{T_n}}}}} \right\}} }}{{\sum\limits_{i = 1}^{T_n} {\mathbb{I}\left\{ {{W_i} > {e^{{k_{T_n}} + 1}}} \right\}} }} - \frac{{\bar F\left( {{e^{{k_{T_n}}}}} \right)}}{{\bar F\left( {{e^{{k_{T_n}} + 1}}} \right)}}} \right).
    \end{equation*}
    converges weakly to a centred normal distribution with variance \({e^{2\beta
                }}\left( {{e^\beta } - 1} \right)\). 
\end{lemma}

\subsection{Proof of Lemma \ref{lemma:dress_and_rotzen}}

For this proof, we will reuse the notation we utilized in the proof of Lemma
\ref{weakly_convergence_ratio} and will add the following definitions: $q_n =nv_n$,
${y_n} = \sqrt {{{{v_n}}}/{{(n{w_n}^2)}}}$ and
\begin{align*}
    U_{n,k} & =\sum\limits_{i = 1}^n {\left( {{\phi _k}\left( {{W_{n,i}}} \right) - \E\left[{\phi _k}\left( {{W_{n,i}}} \right)\right]} \right)}.
\end{align*}

By equation \eqref{eq:linearization} and Lemma \ref{dress_and_rotzen}
\begin{equation*}
    \sqrt{q_n}\left({\widehat \lambda _n} - {\lambda _n}\right) ={y_n}\sqrt{q_n}\left( {{{\widetilde Z}_n}\left( {{\phi _1}} \right) - {\lambda _n}{{\widetilde Z}_n}\left( {{\phi _2}} \right)} \right) + {o_P}\left( 1\right),
\end{equation*}
and \({( {{{\widetilde Z}_n}\left( {{\phi _k}} \right)} )_{1 \leqslant k
                \leqslant 2}}\) converges weakly to a centred normal distribution. Using that
$\lambda_n\to e^\beta$ and $y_n\sim{e^\beta}/{\sqrt{q_n}}$, this implies that
\begin{equation}
    \frac{\sqrt{q_n}\left({\widehat \lambda _n} - {\lambda _n}\right)}{e^\beta} = {{{\widetilde Z}_n}\left( {{\phi _1}} \right) - e^\beta{{\widetilde Z}_n}\left( {{\phi _2}} \right)} + {o_P}\left( 1\right).
\end{equation}

Take $n_0=\left\lfloor \theta n \right\rfloor$, and $V_n=\sqrt{q_n}({{{\widetilde
                Z}_n}\left( {{\phi _1}} \right) - e^\beta{{\widetilde Z}_n}\left( {{\phi _2}}
        \right)})=U_{n,1}-e^\beta U_{n,2}$, then
\begin{equation*}
    \frac{\sqrt{q_{T_n}}}{e^\beta}\left({\widehat \lambda_{T_n}} - {\lambda_{T_n}}\right) = \left( \frac{V_{n_0}}{\sqrt{q_{n_0}}} + \frac{V_{T_n}-V_{n_0}}{\sqrt{q_{n_0}}} \right)\sqrt{\frac{q_{n_0}}{q_{T_n}}} + o_P(1).
\end{equation*}

By Lemma \ref{weakly_convergence_ratio} and our assumption about the convergence in
probability of ${T_n}/{n}$, we have that ${V_{n_0}}/{\sqrt{q_{n_0}}}$ converges in
distribution to a centred Normal random variable with variance $e^{2\beta }( {e^\beta
        } - 1)$ and ${q_{n_0}}/{q_{T_n}}$ converges in probability to $1$. Therefore, if we
show that
\begin{equation}\label{eq:negligible_in_proba}
    \frac{V_{T_n}-V_{n_0}}{\sqrt{q_{n_0}}}\mathop  \to 0
\end{equation}
in probability then our lemma will be proved by two successive applications of Slutsky's theorem.

Notice that
$V_{T_n}-V_{n_0}=U_{T_n,1}-U_{n_0,1}-e^\beta\left(U_{T_n,2}-U_{n_0,2}\right)$, hence,
if we show that ${(U_{T_n,1}-U_{n_0,1})}/{\sqrt{q_{n_0}}}$ and
${(U_{T_n,2}-U_{n_0,2})}/{\sqrt{q_{n_0}}}$ converge to 0 in probability, then
\eqref{eq:negligible_in_proba} will be proved. Given that the proofs of both
convergences are analogous, we will only demonstrate the first one.

Let $\epsilon>0$ be fixed, and set $n_1=\left\lfloor
    n_0\left(1-{\epsilon^3}/{32}\right) \right\rfloor+1, n_2=\left\lfloor
    n_0\left(1+{\epsilon^3}/{32}\right) \right\rfloor$, then
\begin{equation}
    \P\left(\left| U_{T_n,1}-U_{n_0,1} \right|>\epsilon \sqrt{q_{n_0}}\right)\leq I_{n,1}+I_{n,2},
\end{equation}
where,
\begin{align*}
    I_{n,1} & = \P\left( T_n\notin[n_1,n_2] \right),                                                                            \\
    I_{n,2} & = \P\left(\left\{ \left| U_{T_n,1}-U_{n_0,1} \right|>\epsilon \sqrt{q_{n_0}} \right\}\cap T_n\in[n_1,n_2]\right).
\end{align*}

The convergence in probability of ${T_n}/{n}$ to $\theta$ implies that there exists
$N_1$ such that $I_{n,1}<{\epsilon}/{2}$ for all $n\geq N_1$, hence,
\begin{equation}\label{eq:probability_reminder}
    \forall n\geq N_1\quad\P\left(\left| U_{T_n,1}-U_{n_0,1} \right|>\epsilon \sqrt{q_{n_0}}\right)\leq \frac{\epsilon}{2}+I_{n,2}.
\end{equation}

To bound the second term on the right-hand side of the previous display, observe that
$I_{n,2}$ is smaller than
\begin{equation*}
    \P\left(\max_{n_1\leq j \leq n_0} \left| U_{j,1}-U_{n_0,1} \right|>\epsilon\sqrt{q_{n_0}}\right) + \P\left(\max_{n_0< j \leq n_2} \left| U_{j,1}-U_{n_0,1} \right|>\epsilon\sqrt{q_{n_0}}\right).
\end{equation*}

We just need to focus on the case $n_0<j\leq n_2$ because the other will be
analogous. To ease the notation, for any $a<b$, we will write $\bar{F}_{a}^{b}$
instead of $\bar{F}\left(u_{a}\right)-\bar{F}\left(u_{b}\right)$ and we will use
$\bar{F}_a$ to denote $\bar{F}(u_{a})$. Let $\mathcal{C}_n$ be the set $\{\max_{n_0<
        j \leq n_2} | U_{j,1}-U_{n_0,1} |>\epsilon\sqrt{q_{n_0}}\}$. We can write the
difference $U_{j,1}-U_{n_0,1}$ as
\begin{align}
    U_{j,1}-U_{n_0,1} & =\sum\limits_{i = 1}^j \left(\mathbb{I}\left\{ {{W_i} > u_j} \right\}-\bar{F}_j\right) - \sum\limits_{i = 1}^{n_0} \left(\mathbb{I}\left\{ {{W_i} > u_{n_0}} \right\}-\bar{F}_{n_0}\right)\nonumber                                                   \\
                      & = \sum\limits_{i = n_0+1}^j \left(\mathbb{I}\left\{ {{W_i} > u_{n_0}} \right\}-\bar{F}_{n_0}\right) - \sum\limits_{i = 1}^{j} \left(\mathbb{I}\left\{ {{W_i} \in (u_{n_0},u_j]} \right\}-\bar{F}_{n_0}^{j}\right)\label{eq:difference_decomposition}.
\end{align}

Suppose for the moment that
\begin{equation}\label{eq:convergence_to_zero_in_proba}
    \frac{\underset{n_0<j \leq n_2}{\max}\left|\sum\limits_{i = 1}^{j} \left(\mathbb{I}\left\{ {{W_i} \in (u_{n_0},u_j]} \right\}-\bar{F}_{n_0}^{j}\right)\right|}{\sqrt{q_{n_0}}}\to 0
\end{equation}
in probability.

Let $\mathcal{G}_n$ be the event $\{\max_{n_0< j \leq n_2}{|\sum_{i = 1}^{j}
    (\mathbb{I}\left\{ {{W_i} \in (u_{n_0},u_j]}
            \right\}-\bar{F}_{n_0}^{j})|}\leq{\epsilon\sqrt{q_{n_0}}}/{2}\}$. By
\eqref{eq:convergence_to_zero_in_proba}, we can find $N_2$ such that $\P\left(
    \mathcal{G}_n \right)\geq1-{\epsilon}/{8}$ for all $n\geq N_2$. Therefore, for all
$n\geq N_2$ we have
\begin{equation*}
    \P\left(\mathcal{C}_n\right) \leq \P\left( \left\{ \max_{n_0< j \leq n_2} \left| U_{j,1}-U_{n_0,1} \right|>\epsilon\sqrt{q_{n_0}} \right\}\cap\mathcal{G}_n \right)+ \frac{\epsilon}{8}.
\end{equation*}

Using \eqref{eq:difference_decomposition}, we obtain that $\{ \max_{n_0 < j \leq n_2}
    \left| U_{j,1}-U_{n_0,1} \right|>\epsilon\sqrt{q_{n_0}} \}\cap\mathcal{G}_n$ is
contained in the event
\begin{equation*}
    \left\{\underset{n_0< j \leq n_2}\max \left|\sum\limits_{i = n_0+1}^j \left(\mathbb{I}\left\{ {{W_i} > u_{n_0}} \right\}-\bar{F}_{n_0}\right)\right|>\frac{\epsilon\sqrt{q_{n_0}}}{2}\right\}.
\end{equation*}

By Kolmogorov inequality,
\begin{equation*}
    \P\left(\underset{n_0< j \leq n_2}\max \left|\sum\limits_{i = n_0+1}^j \left(\mathbb{I}\left\{ {{W_i} > u_{n_0}} \right\}-\bar{F}_{n_0}\right)\right|>\frac{\epsilon\sqrt{q_{n_0}}}{2}\right) \leq \frac{4\left(n_2-n_0\right)\bar{F}_{n_0}}{\epsilon^2 n_0\bar{F}_{n_0}}.
\end{equation*}

The right-hand side of the previous equation equals ${4\left(\left\lfloor
            n_0\left(1+{\epsilon^3}/{32}\right) \right\rfloor-n_0\right)}/{(\epsilon^2 n_0)}$,
and that is smaller than $\epsilon/8$. Hence, $\P\left(\mathcal{C}_n\right) \leq
    \frac{\epsilon}{4}$ for all $n\geq N_2$. In a similar fashion, we can find $N_3$ such
that $\P\left(\max_{{n_1\leq j \leq n_0}} \left| U_{j,1}-U_{n_0,1}
    \right|>\epsilon\sqrt{q_{n_0}}\right)\leq {\epsilon}/{4}$ for $n\geq N_3$. This shows
that $I_{n,2}\leq{\epsilon}/{2}$ for $n\geq\max(N_2,N_3)$. Combining this with
equation \eqref{eq:probability_reminder}, proofs \eqref{eq:negligible_in_proba}.

To finish, we proceed with the proof of \eqref{eq:convergence_to_zero_in_proba}. Let
$\delta>0$ be fixed. Without loss of generality, assume that $1+{\epsilon^3}/{8}<2$
and ${\epsilon^3}/{8}<D^{\prime}$. Denote by $\mathcal{H}_{n,\delta}$ the event
$\{{\max_{n_0< j \leq n_2}}|\sum_{i = 1}^{j} (\mathbb{I}\{ {{W_i} \in (u_{n_0},u_j]}
    \}-\bar{F}_{n_0}^{j})|>\delta \sqrt{q_{n_0}} \}$, then,
\begin{equation}\label{eq:maximum_inclusion}
    \mathcal{H}_{n,\delta}\subseteq\bigcup_{j=n_0+1}^{n_2}{\left\{ \left|\sum\limits_{i = 1}^{j} \left(\mathbb{I}\left\{ {{W_i} \in (u_{n_0},u_j]} \right\}-\bar{F}_{n_0}^{j}\right)\right|>\delta\sqrt{q_{n_0}} \right\}}.
\end{equation}

By Chebyshev's inequality,
\begin{equation*}
    \P\left( \left|\sum\limits_{i = 1}^{j} \left(\mathbb{I}\left\{ {{W_i} \in (u_{n_0},u_j]} \right\}-\bar{F}_{n_0}^{j}\right)\right|>\delta\sqrt{q_{n_0}} \right)\leq \frac{2\bar{F}_{n_0}^{j}}{\delta^2 \bar{F}_{n_0}}=\frac{2}{\delta^2}\left(1-\frac{\bar{F}_j}{\bar{F}_{n_0}}\right).
\end{equation*}

Combining this with \eqref{eq:maximum_inclusion} and \AssumpVelocity, we obtain
\begin{align*}
    \P\left( \mathcal{H}_{n,\delta} \right) & \leq \frac{2}{\delta^2}\sum_{j=n_0+1}^{n_2}{\left(1-\frac{\bar{F}_j}{\bar{F}_{n_0}}\right)} \leq \frac{2}{\delta^2}\left(n_2-n_0 - \frac{\sum_{j=n_0+1}^{n_2}\bar{F}_j}{\bar{F}_{n_0}}\right)       \\
                                            & \leq \frac{2}{\delta^2}\left(n_2-n_0 - \frac{\left(n_2-n_0\right){\bar{F}_{n_0}+o\left({\bar{F}_{n_0}}\right)}}{{\bar{F}_{{n_0}}}}\right)                                        = o\left(1\right),
\end{align*}
which completes the proof of \eqref{eq:convergence_to_zero_in_proba}.

\section{Functional version of the limit distributions}\label{sec:funclimit}

Let $n\geq 1$, and define the step function
\begin{equation}\label{eq:process}
    \sigma_{n}(B):\; t\geq 0 \mapsto \frac{\Sigma_{\lfloor {nt} \rfloor}(B)}{ n^{\beta}L_{v,D}(n)}.
\end{equation}

The result stated below describes the asymptotic distribution of the process
$\sigma_n(B)$. See Theorem 17.4.4 in \cite{Meyn2009} for an analogous result in the
positive recurrent case. We denote by $M_{\beta}=(M_{\beta}(t))_{t\geq 0}$ the
Mittag-Leffler process with parameter $\beta\in (0,1)$, defined by:
\begin{align*}
    \mathbb{E}\left[ \left( M_{_\beta }( 1) \right)^m \right] & = \frac{{m!}}{{\Gamma \left( {1 + m \beta } \right)}},\quad \text{for all } m \geqslant 0, \\
    {M_\beta }\left( t \right)                                & \mathop  = \limits^d {t^\beta }{M_\beta }\left( 1 \right),\quad \text{for all } t\geq 0.
\end{align*}

The characteristic functions describing the marginal distributions are given by (see
3.39 in \cite{Tjostheim-2001})
\begin{equation}\label{eq:characteristic-fn-mittag-leffler}
    \mathbb{E}\left[ {{e^{i\zeta {M_\beta }\left( t \right)}}} \right] = \sum\limits_{k = 0}^{ + \infty } {\frac{{{{\left( {i\zeta {t^\beta }} \right)}^k}}}{{\Gamma \left( {1 + k\beta } \right)}}} ,\quad \zeta  \in \mathbb{R},\;t \geqslant 0.
\end{equation}

\begin{theorem}[Functional Limit Theorem]\label{thm:funclimit}
    Let $\beta\in(0,1)$ and $\nu$ be any probability distribution. Suppose that the
    chain $X$ is $\beta$-regular and \(B\) is a Harris set. Then, as \(n\to\infty\), we have:
    \begin{equation*}
        \sigma_{n}(B) \Rightarrow \mu(B) \Gamma \left(1+\beta \right){M_\beta }\text{ in } \mathbb{P}_{\nu}\text{-distribution},
    \end{equation*}
    in the sense of Skorokhod topology.
\end{theorem}

\begin{remark} Let \(Y=1/Z_\beta^\beta\) where \(Z_\beta\) is as in Theorem \ref{thm:weaklimit}. By virtue of equation (8.3) in page 453 of \cite{Feller1971}, the Laplace transform of \(Y\) is
    \begin{equation*}
        \mathbb{E}\left[ {{e^{ - sY}}} \right] = \sum\limits_{k = 0}^{ + \infty } {\frac{{{{\left( { - \Gamma \left( {1 + \beta } \right)s} \right)}^k}}}{{\Gamma \left( {1 + k\beta } \right)}}},
    \end{equation*}
    which equals the Laplace transform of \(\Gamma \left( {1 + \beta } \right){M_\beta
    }\left( 1 \right)\),\; \textit{cf} \eqref{eq:characteristic-fn-mittag-leffler}.
\end{remark}

\subsection{Proof of Theorem \ref{thm:funclimit}}

Let \(L_s\left(n\right)\) be such that $ L\left( n \right) = {(\Gamma \left( {1 -
        \beta } \right){L_s}\left( n \right))}^{-1}$ with \(L\) as defined in
(\ref{eq:Zipf}). Observe that,
\begin{equation*}
    {G_{\nu ,D}}\left( n \right) = \frac{1}{{\mu \left( D \right)}}\sum\limits_{k = 1}^n {{\P_\nu }\left( {{X_k} \in D} \right)}  = \frac{1}{{\mu \left( D \right)}}{\E_\nu }\left[ {\sum\limits_{k = 1}^n {\mathbb{I}\left\{ {{X_k} \in D} \right\}} } \right].
\end{equation*}

By Lemma 3.1 and Definition 3.2 of \cite{Tjostheim-2001} ${\E_\nu }\left[ {\sum_{k =
        1}^n {\mathbb{I}\left\{ {{X_k} \in D} \right\}} } \right]$ is asymptotically
equivalent to $\Gamma \left( {1 + \beta } \right)^{-1}{n^\beta }\mu \left( D
    \right){L_s}\left( n \right),$ therefore,
\begin{equation}\label{green_function_limit}
    {G_{\nu ,D}}\left( n \right) \sim \frac{1}{{\Gamma \left( {1 + \beta } \right)}}{n^\beta }{L_s}\left( n \right).
\end{equation}

Let \(u\left(n\right)=n^\beta L_s\left(n\right)\), and define the process
\begin{equation*}
    {{\mathbf{T}}_{n,B}} = {\left\{ {\frac{{{S_{\left\lfloor {nt} \right\rfloor }}\left( B \right)}}{{u\left( n \right)}}} \right\}_{t \geqslant 0}} = \frac{{{G_{\nu ,D} }\left( n \right)\mu \left( B \right)}}{{u\left( n \right)}}{{\mathbf{S}}_{n,B}}.
\end{equation*}

By Lemma 3.6 and Theorem 3.2 of \cite{Tjostheim-2001}, \({{\mathbf{T}}_{n,B}}\)
converges weakly on the Skorokhod topology to the process $\mu\left(B\right)M_\beta$
and by \eqref{green_function_limit}, ${{{G_{\nu ,D}}\left( n \right)}}/{{u\left( n
            \right)}} \to {1}/{{\Gamma \left( {1 + \beta } \right)}}$ which completes the proof.

\section{Nummelin's splitting technique and construction of pseudo regeneration blocks}\label{sec:pseudo_regeneration}

Nummelins splitting technique relies crucially on the notion of \textit{small set}: a
set $K\in\mathcal{E}$ is said to be \textit{small} if there exist
$m\in\mathbb{N}^{\ast}$, $\delta>0$ and a probability measure $\Phi$ supported by
$\smallSet$ such that
\begin{equation} \label{minor}
    \forall \left(x,B\right)\in K\times \E,\;\; \Pi_{m}\left(x,B\right)\geq\delta\Phi\left(B\right).
\end{equation}

We refer to \eqref{minor} as the minorization condition $\mathcal{M}
    (m,K,\delta,\Phi).$ Recall that accessible small sets always exist for
$\psi$-irreducible chains: any set $B\in\mathcal{E}$ such that $\psi(B)>0$ contains
such a set, see \cite{JainJam}. Suppose that $X$ satisfies
$\mathcal{M}=\mathcal{M}(m,K,\delta,\Psi)$ for $K\in\mathcal{E}$ s.t. $\psi(K)>0.$
Rather than replacing the initial chain $X$ by the chain
$\{(X_{nm},...,X_{n(m+1)-1})\}_{n\in\mathbb{N}}$, we suppose $m=1$. The sample space
is expanded so as to define a sequence $(Y_{n})_{n\in\mathbb{N}}$ of independent
Bernoulli r.v.'s with parameter $\delta$ by defining the joint distribution,
$\mathbb{P}_{\nu,\mathcal{M}}$ whose construction relies on the following
randomization of the transition probability $\Pi$ each time the chain hits
$\smallSet$. Note that it occurs with probability one, since the chain is Harris
recurrent and $\psi(K)>0$. If $X_{n}\in K$, and if $Y_{n}=1$ (this occurs with
probability $\delta\in\left] 0,1\right[ $), then $X_{n+1}\sim \Phi$, while, if
$Y_{n}=0$, we have $X_{n+1}\sim (1-\delta)^{-1}(\Pi(X_{n},.)-\delta\Phi(.))$. Let
$\text{Ber}_{\delta}$ be the Bernoulli distribution with parameter $\delta$. The
\textit{split chain} $\{(X_{n},Y_{n})\}_{n\in\mathbb{N}}$ is valued in $E\times \{
    0,\; 1\} $ and has transition kernel $\Pi_{\mathcal{M}}$%
\begin{itemize}
    \item
          for any $x\notin K$, $B\in\mathcal{E}$, $b$ and $b'$ in $\left\{ 0,1\right\} ,$%
          \begin{equation*}
              \Pi_{\mathcal{M}}\left(  \left(  x,b \right)  ,B\times \left\{
              b'\right\}  \right)  =\Pi\left(  x,B\right)  \times \text{Ber}%
              _{\delta}(b'),
          \end{equation*}
    \item for any $x\in K$, $B\in\mathcal{E}$, $b'$ in $\left\{ 0,1\right\}$,
          \begin{equation*}
              \left\{
              \begin{array}{ccl}
                  \Pi_{\mathcal{M}}\left(  \left(  x,1\right)  ,B\times \left\{  b'\right\}  \right) & = & \Phi(B)\times \text{Ber}_{\delta}(b'),         \\
                  \Pi_{\mathcal{M}}\left(  \left(  x,0\right)  ,B\times \left\{  \beta^{\prime
                  }\right\}  \right)                                                                 & = & (1-\delta)^{-1}(\Pi\left(  x,B\right)  -\delta
                  \Phi(B))\times \text{Ber}_{\delta}(b').
              \end{array}
              \right.
          \end{equation*}
\end{itemize}

The key point of the construction relies on the fact that $A_K=K\times \{1\}$ is an
atom for the bivariate chain $X^{\mathcal{M}}=(X,Y)$, which inherits all its
communication and stochastic stability properties from $X$.

Recall also that conditions of type \eqref{minor} can be replaced by Foster-Lyapunov
drift conditions that are much more tractable in practice, see \textit{e.g.} Chapter
11 in \cite{Meyn2009}. The construction above permits the extension of probabilistic
results established for regenerative chains to general recurrent Harris chains. In
particular, we have the following result, presented in \cite[pp 19]{Chen1999}.

\begin{proposition}\label{prop:split_chain_beta_regular}
    Let $X$ be a $\beta$-regular chain, with $\beta\in [0,1]$. Suppose that condition $\mathcal{M}=\mathcal{M}(1,K,\delta,\Psi)$ is fulfilled,  then, the split chain $X^{\mathcal{M}}$ is $\beta$-regular.
\end{proposition}

Hence, $X$'s regularity index $\beta$ is the regular variation index of the
conditional survivor function of the hitting time $\tau_{A_K}=\inf\{n\geq 1:\;
    (X_n,Y_n)\in A_K\}$ given $(X_0,Y_0)\in A_K$. However, the return times to $A_K$ are
not observable, just like the sample path of the Nummelin extension, and cannot be
straightforwardly exploited from a statistical perspective. We shall explain in
subsection \ref{subsec:extension_nummelin} how estimators tailored to the
regenerative case can be nevertheless extended to the pseudo-regenerative case in
practice by means of the \textit{plug-in} approximation procedure originally proposed
in \cite{Bertail2006} in the positive recurrent case.

\subsection{Obtaining samples form the split chain}\label{sec:samples_split_chain}

Proposition \ref{prop:split_chain_beta_regular} and the algorithmic construction
described after Eq. \eqref{minor} guarantee that if the chain satisfies the
minorization condition $\mathcal{M}=\mathcal{M}(1,K,\delta,\Psi)$, and $K,\delta$ and
$\Psi$ are known, then we can generate samples of the split chain, which is atomic
and has the same $\beta$ as the original chain. Assume the existence of a
$\sigma$-finite measure $\lambda$ of reference on $(E,\mathcal{E})$ that dominates
the conditional probability measures $\Pi(x,\; dy)$, $x\in E$, and the initial
distribution $\nu$: $\Pi(.,dy)=\pi(.,y)\lambda(dy)$ and $\nu(dy)=g(y)\lambda(dy)$.
Notice incidentally that the measure $\Psi$ involved in $\mathcal{M}$ is then
absolutely continuous w.r.t. $\lambda$ as well: $\Psi(dy)=\psi(y)\lambda(dy)$ and
then $\pi\left(x,y\right)\geq \delta \psi\left(y\right)$ for all $(x,y)\in K^2$. As
shown in Section 3.2 in \cite{Bertail2006}, given $X^{(n+1)}=\left(X_1, \ldots,
    X_{n+1}\right)$, samples from the distribution of $Y^{(n)}=\left(Y_1, \ldots\right.$,
    $\left.Y_n\right)$ can be obtained as follows. From $i=1$ to $n$, the r.v. $Y_i$ is
drawn from a Bernoulli distribution with parameter $\delta$, unless $X$ hits the
small set $\smallSet$ at time $i$: in the latter case, $Y_i$ is drawn from a
Bernoulli distribution with parameter $\delta \psi\left(X_{i+1}\right) /
    \pi\left(X_i, X_{i+1}\right)$. Given that $A_K=K\times \{1\}$ is an atom for the
split chain, and the statistics under study in this paper only depends on the size of
the regeneration blocks, sampling $Y_i$ when $X_i\in K$ is sufficient here. The
accuracy of the estimator improves as the (random) number of samples (the number of
regeneration blocks namely) increases. This number is influenced by the size of the
chosen small set and how frequently the chain visits it in a finite-length
trajectory. It is also affected by the sharpness of the lower bound in the
minorization condition. Essentially, there is a trade-off that can be described as
follows. Increasing the size of the small set $\smallSet$ used for constructing the
pseudo-blocks naturally increases the number of time points that could determine a
block (or a cut in the trajectory). However, it also reduces the probability of
cutting the trajectory, as the uniform lower bound for $\pi(x,y)$ over $\smallSet^2$
then decreases. This suggests a criterion for selecting the small set $\smallSet$:
choose a small set that maximizes the maximum expected number of data blocks given
the trajectory, that is
\begin{equation*}
    N_n(\smallSet)=\mathbb{E}_\nu\left[\sum_{i=1}^n \mathbb{I}\left\{X_i \in \smallSet, Y_i=1\right\} \mid X^{(n+1)}\right].
\end{equation*}
In Section 3.6 of \cite{Bertail2006}, a data-driven approach to select the small set is
proposed for the cases where the chain takes real values. The idea relies on the fact
that, in many cases, for a well-chosen $x_0$ and $\epsilon$ small enough, certain intervals
$V_{x_0,\epsilon}=[x_0-\epsilon,x_0+\epsilon]$ are small sets, with the minorization
measure $\Psi$ being the Lebesgue measure on $V_{x_0,\epsilon}$. Given a point $x_0$
(generally taken as the mean or the median of the $X_i$'s), the proposed algorithm finds
the value of $\epsilon$ that maximizes the expected number of regeneration blocks, that
is
\begin{equation*}
    N_n\left(V_{x_0,\epsilon}\right)=\frac{\delta\left(V_{x_0,\epsilon}\right)}{2\epsilon} \sum_{i=1}^n
    \frac{\mathbb{I}\left\{\left(X_i, X_{i+1}\right) \in V_{x_0,\epsilon}^2\right\}}{\pi\left(X_i, X_{i+1}\right)},
\end{equation*}
where $\delta\left(V_{x_0,\epsilon}\right)=2\epsilon\inf_{(x,y)\in V^2_{x_0,\epsilon}}\pi(x,y)$.
Then, the samples of the split chain can be obtained by following the procedure
described at the begining of this subsection with
$K=V_{x_0,\epsilon}$, $\delta=2\epsilon\inf_{(x,y)\in V^2_{x_0,\epsilon}}\pi(x,y)$
and $\psi(y)=1/(2\epsilon)$.

\end{appendices}

\end{document}